\documentclass[11pt,reqno]{amsart} 
\usepackage[utf8]{inputenc} 
\usepackage[left=2.5cm,right=2.5cm,top=2.5cm,bottom=2.5cm]{geometry}
\usepackage{graphicx} 
\usepackage{epstopdf}
\usepackage{color}
\usepackage{cite}
\usepackage{mathtools}
\usepackage{xcolor}
\usepackage{tikz}

\makeatletter

\renewcommand\subsubsection{\@secnumfont}{\bfseries}%
\renewcommand\subsubsection{\@startsection{subsubsection}{3}
  \z@{.5\linespacing\@plus.7\linespacing}{-.5em}%
  {\normalfont\bfseries}}
  
  \makeatother

\usepackage{array} 
\usepackage{paralist} 
\usepackage{verbatim} 
\usepackage{subfig} 

\usepackage{esint}
\usepackage{amsfonts}
\usepackage{amsmath}
\usepackage{amsthm}
\usepackage{amssymb}
\usepackage{hyperref}
\usepackage{dsfont}

\hypersetup{
	colorlinks=true,
	linkcolor=blue, 
	citecolor=blue,
	filecolor=blue,
	urlcolor=blue,
}

\newcommand{\mel}{\MoveEqLeft}

\newtheorem{theorem}{Theorem}[section]
\newtheorem{proposition*}{Proposition\textsuperscript{*}}

\newtheorem{corollary*}{Corollary\textsuperscript{*}}

\newtheorem{lemma}[theorem]{Lemma}

\theoremstyle{definition}

\newtheorem{example*}{Example\textsuperscript{*}}

\numberwithin{equation}{section}

\usepackage{bbm}

\def\Limes#1#2 {\lim\limits_{#1\rightarrow #2}}

\def\R{\mathbb{R}}

\def\Z{\mathbb{Z}}

\newcommand{\Chi}{\mathcal{X}}

\def\XXint#1#2#3{{\setbox0=\hbox{$#1{#2#3}{\int}$ }
\vcenter{\hbox{$#2#3$ }}\kern-.59\wd0}}

\renewcommand{\epsilon}{\varepsilon}

\def\dfrac{\frac{\mathrm{d}}{\mathrm{d}t}}

\DeclareMathOperator{\BS}{BS}

\def\scalar#1#2{\langle #1,#2 \rangle}
\def\de{\partial}

\def\dx{\,\mathrm{d}x}

\def\dt{\,\mathrm{d}t}

\newcommand\minus\backslash

\newcommand\lan\langle
\newcommand\ran\rangle

\newcommand{\supp}{\operatorname{supp}}

\newcommand{\dd}{{\,\mathrm d}}

\DeclareMathOperator\dist{dist}

\newcommand{\norm}[1]{\left\lVert#1\right\rVert}

\renewcommand\leq\leqslant
\renewcommand\geq\geqslant

\title[Vortex-Wave uniqueness ]{A Uniqueness theorem for the Vortex-Wave system with non-constant vorticity near the point vortex}

 \author[D. Meyer]{David Meyer}
 \address{ \vspace{-0.4cm}
\newline 
\textbf{{\small David Meyer}} 
\vspace{0.15cm}
\newline \indent Max-Planck Institute for Mathematics in the Sciences, Leipzig, Germany}
 \email{david.meyer@mis.mpg.de}

\keywords{}
\subjclass[2020]{}

\begin{document}
\begin{abstract}
We study the uniqueness of the vortex wave system, i.e.\ the 2D incompressible Euler equations with vorticity consisting of a point vortex and a bounded part. We show that if the initial vorticity and point vortex location $(w_0,X_0)$  satisfy a constraint of the type $|w_0(x)-w_0(X_0)|=O(|x-X_0|^\alpha)$ for $\alpha> 2$ and the vorticity is in $L^\infty$, then solutions are unique, globally in time, even though the velocity is not log-Lipschitz.
\end{abstract}

\maketitle

\section{Introduction and main results}

The 2D incompressible Euler equations are one of the most fundamental models of fluid dynamics.
In vorticity form, these equations read as \begin{align}
    \de_t\omega+(u\cdot \nabla)\omega=0\quad \text{and }\quad u=\nabla^\perp\Delta^{-1}\omega \label{oiler}
\end{align}  
where $\omega:\R^2\rightarrow \R$ is the vorticity and $u:\R^2\rightarrow \R^2$ is the velocity.

We are concerned with solutions of a specific form, given by the so-called \textit{vortex-wave system}. This case consists of solutions of the form $\omega=w+\delta_{X(t)}$, i.e.\ a regular part and a point vortex.
The system is a model for the interaction of a strongly concentrated vortex, idealized into a point vortex, with the surrounding fluid.

For such data, the velocity is given by \begin{align*}
    u(x)=(\nabla^\perp \Delta^{-1} w)(x)+\frac{(x-X(t))^\perp}{2\pi|x-X(t)|^2},  
\end{align*}
where $y^\perp=(-y_2,y_1)$, thanks to the explicit form of the full-space inverse Laplacian.

In general, this has a singularity at $x=X(t)$ and the transport equation for the Dirac-Delta at $X(t)$ needs to be somewhat reinterpreted. Since the velocity of the point-vortex is rotationally symmetric, it should not transport the point-vortex itself and the system should evolve according to the equations \begin{align}
    \de_t w+(u\cdot\nabla) w=0, \qquad u=\nabla^\perp \Delta^{-1} w+\frac{(x-X(t))^\perp}{2\pi|x-X(t)|^2},\qquad \dfrac X(t)=(\nabla^\perp \Delta^{-1} w)(X(t)).\label{vw} 
\end{align}
Even with these modifications, such data is still far too singular to fit into the usual well-posedness framework for the 2D Euler equations.

It is relatively straightforward to show that for initial data $w_0\in L^1\cap L^\infty$ there exists a weak solution (see definition \eqref{def sol} below), which was first done by Marchioro and Pulvirenti in \cite{marchioro1991vortex}, see also \cite{lopes2011existence} for the $L^p$-case.

Uniqueness of such solutions is a much more difficult and largely unresolved problem. Indeed, almost all existing uniqueness proofs for the 2D Euler equations require the velocity to be log-Lipschitz and that is not the case here.

This limitation is also sharp, and it can be shown
 that for measure- or even $L^p$-valued vorticities (leading to velocities which are only Sobolev-regular), solutions to the Euler equations \eqref{oiler} with an additional forcing term on the right-hand side are non-unique, see e.g.\ \cite{vishik2018a, vishik2018b,grotto2022burst, brue2026flexibility}.
However, in this specific problem, the singular part of the solution(s) has far more structure than a "generic" measure-valued vorticity, and therefore one might hope that uniqueness still holds.


In fact, it can be shown \cite{starovoitov1994uniqueness,lacave2009uniqueness} that if the initial data $w_0\in L^1\cap L^\infty$ is constant in a neighbourhood of the point-vortex, then this property is preserved globally in time and leads to (global in time) uniqueness of solutions, as log-Lipschitzness of the velocity in a neighborhood of the region where the vorticity is non-constant is sufficient for uniqueness.

The goal of this paper is to improve this result and to show that uniqueness still holds for non-constant vorticities,
 if the initial vorticity converges faster than second order to a constant at $X(0)$, i.e.\ that there is some $\alpha>2$ such that \begin{align}
\limsup_{x\rightarrow X(0)} \frac{|w_0(x)-w_0(X(0))|}{|x-X(0)|^\alpha}<\infty.\label{main ass}
\end{align}
This contains the case of $w_0$ being constant near $X(0)$ as a special case.

\subsection{Main result}

We say that a given triple $(w,u,X)$ is a solution to the vortex wave system for initial data $w_0,X(0)$ if the first equation in  \eqref{vw} holds distributionally, i.e.\ for any $\phi\in C_c^\infty(\R^2\times \R_{\geq 0})$ it holds that \begin{align}
\begin{drcases}
&\int_0^\infty\int_{\R^2} w_t(x)\de_t\phi(t,x)+u_t(x)\cdot \nabla\phi(t,x) w_t(x) \dx \dt=-\int_{\R^2} w_0(x)\phi(0,x)\dd x\qquad\qquad\quad\\
& u_t(x)=(\nabla^\perp \Delta^{-1} w_t)(x)+\frac{(x-X(t))^\perp}{2\pi|x-X(t)|^2}\\
&\dfrac X(t)=(\nabla^\perp \Delta^{-1} w_t)(X(t)).
\end{drcases}\label{def sol}
\end{align}
Note that this is not equivalent to being a weak solution to the 2D Euler equations in velocity form in the usual sense (e.g.\ the velocity is not in $L^2$). Still, solutions with $w_t\in L^1\cap L^\infty$ always exist for $w_0\in L^1\cap L^\infty$ \cite{marchioro1991vortex} and can also be recovered as suitable limits of classical solutions \cite{bjorland2011vortex}. 
Furthermore, $w\in L^1\cap L^\infty$ guarantees that $\nabla^\perp\Delta^{-1} w_t$ is log-Lipschitz (see Lemma \ref{ll lemma app}) and that the ODE for $X$ is uniquely solvable for a given $w$.

A detailed discussion of the equivalence of this solution concept with Lagrangian formulations can be found in \cite{lacave2009uniqueness}.

Our main theorem is then the following: \begin{theorem}\label{T1}
Suppose $X(0)\in \R^2$ is given and $w_0\in L^\infty(\R^2)$ has compact support and fulfills the assumption \eqref{main ass}.

Then the solution $w\in L^\infty([0,\infty),L^1(\R^2)\cap L^\infty(\R^2))$ to the vortex-wave system with this initial data in the sense of \eqref{def sol} is unique.
\end{theorem}

Here the assumption of compact support of $w_0$ merely serves to simplify the proof, one can prove the same result also for $w_0\in L^1\cap L^\infty$ at the price of a few additional calculations.\newline
In principle, the argument works also for multiple point vortices or bounded domains with a sufficiently regular boundary if there are no collisions of the point vortices with each other or the boundary.\smallskip

Let us say a few more words about the existing literature:\newline
The vortex-wave system is mostly used to study the interaction of a concentrated vortex with the surrounding vorticity, see e.g.\ \cite{ionescu2022axi,flandoli2026early}.
However, it can also be derived as the vanishing solid limit of fluid-solid systems with non-zero circulation around the solids, see e.g.\ \cite{glass2014motion,glass2016motion,glass2019dynamics}.

Similar models for the SQG or Vlasov-Poisson equations have been studied in \cite{cobb2025existence,marchioro2011cauchy,crippa,pausader2024stability}. 

For the 2D Navier-Stokes equation, global existence and uniqueness for initial data containing point vortices has been established in \cite{gallagher2005uniqueness}. A detailed study of the vanishing viscosity limit for such data has been given in \cite{nguyen2019inviscid}. 

Regarding the uniqueness of the 2D Euler equations for velocities which are not log-Lipschitz or Osgood, there are various results, either relying on the vorticity being constant near the singular part, see e.g.\ \cite{lacave2009uniqueness,lacave2019euler,han2021euler}, specific spiral symmetries \cite{jeong2024logarithmic,elgindi2023wellposedness}, the special structure of singularities induced by rough boundaries \cite{agrawal2020uniqueness,agrawal2025uniqueness}, or the Baire category theorem \cite{galeati20262d}.
 Let us also mention the case of bounded vorticity without decay at $\infty$ as another area in which there is ongoing research, see e.g.\ \cite{elgindi2020symmetries,mcowen2022perfect,cobb2026unbounded} and the references therein.





\subsection{Idea of the proof}
One can eliminate the variable $X$ by rewriting the equations in a frame centered on $X$, i.e.\ by making the change of variables $x\rightarrow x-X$, $\varpi_t=w_t(\cdot -X(t))$, $\bar{u}_t=u_t(\cdot-X(t))$, leading to the equations \begin{align}
\de_t \varpi_t+(\bar u_t\cdot\nabla) \varpi_t=0,\quad \bar u_t(x)=v_t(x)-v_t(0)+\frac{x^\perp}{2\pi|x|^2},\quad v_t=\nabla^\perp \Delta^{-1}\varpi_t.\label{xframe}
\end{align}
The classical argument for the uniqueness of the Euler equations is to show a differential estimate for $\norm{\mathcal{X}_t^1-\mathcal{X}_t^2}_{L^2}$ where $\mathcal{X}_t^1$ and $\mathcal{X}_t^2$ are the flow maps of two different solutions from the same initial data, and to conclude from Gronwall's Lemma that this must be $0$, implying uniqueness, see e.g.\ \cite{marchioro2012mathematical,loeper2006uniqueness}. This normally requires one to estimate $\norm{u_t^1(\mathcal{X}_t^1)-u_t^2(\mathcal{X}_t^1)}_{L^2}$ and $\norm{u_t^2(\mathcal{X}_t^1)-u_t^2(\mathcal{X}_t^2)}_{L^2}$ in an appropriate way.
Closing the estimates for the latter term usually requires the velocity to be Lipschitz or log-Lipschitz, which is not the case here.

Let us explain how to circumvent this issue here:

Somewhat inspired by the approach in \cite{agrawal2025uniqueness}, the main idea is to adapt the usual proof of uniqueness to a more strict metric, in which the velocity of the point vortex is almost Lipschitz in the sense that its symmetric gradient in this metric is bounded (which is enough to control the spread of trajectories).

This metric is \begin{align}
m(x,y):=\max(|x|,|y|)|x_\theta-y_\theta|+\max(|x|,|y|)^{-2}\big||x|-|y|\big|,\label{def m}
\end{align}
where $x_\theta$ and $y_\theta$ are standard polar coordinates.
The remaining velocity $v_t-v_t(0)$ is clearly not Lipschitz or Log-Lipschitz in this metric, though.
However, close to the point vortex one expects that the actual impact of the velocity field $v-v(0)$ on the flow should be much smaller than the instantaneous velocity $v_t-v_t(0)$, since there should be massive cancellations due to the fact that the rotation of a particle around the point vortex is much faster than the variations of $v_t-v_t(0)$ and hence the velocity $v_t-v_t(0)$ should be "averaged" along rotations around the point vortex which makes it much smaller.

This idea is implemented by
taking a metric which is equivalent to $m$, depending on $\varpi_t$, for which the relevant time derivatives under the flow are much smaller. More concretely, we will look at something like \begin{align*}
\max(|x|,|y|)|x_\theta-y_\theta|+\max(|x|,|y|)^{-2}\big|d(t,x)-d(t,y)\big|\approx m(x,y),
\end{align*}
where $d(t,x)\approx |x|$ and $\nabla d$ is orthogonal to $\bar u$ near $0$ (see Section \ref{S24}).
If one then tracks the evolution of $d(t,\mathcal{X}_t(x))-d(t,\mathcal{X}_t(y))$, where $\mathcal{X}_t$ is the flow map of $u$, only the time derivative of $d$ itself contributes to the time derivative of this. This time derivative can be made $\lesssim \max(|x|,|y|)^2|x-y|\log(2+\frac{1}{|x-y|})$ if an assumption like \eqref{main ass} holds (see Lemma \ref{det d}).
Requiring some kind of smallness assumption like \eqref{main ass} is natural for such an argument, as the velocity with which $\varpi$ moves diverges near $0$ and hence the contribution of this part to the time derivative of $d$ would be of a bigger order without a smallness condition. The assumption \eqref{main ass} can easily be shown to be propagated under the evolution, see Lemma \ref{decay prop}.

Hence, even though $\dfrac m(\mathcal{X}_t(x),\mathcal{X}_t(y))$ is unbounded, we do still have $m(\mathcal{X}_t(x),\mathcal{X}_t(y))\lesssim |x-y|^{e^{-Ct}}$ and the flow behaves as if its velocity were log-Lipschitz with respect to the metric $m$.

However, this functional depends on $\bar u_t$, and if we have two solutions from the same initial datum whose equality we are trying to prove, this also gives rise to two different realisations of $d$. Hence, instead of looking at $d(t,\mathcal{X}_t^1(x))-d(t,\mathcal{X}_t^2(x))$, we will look at $d^1(t,\mathcal{X}_t^1(x))-d^2(t,\mathcal{X}_t^2(x))$, where $d^1$ and $d^2$ are the two different realisations of $d$ for the two different solutions and add an extra term of the type $|x|^{-2}|d^1(t,\mathcal{X}_t^1(x))-d^2(t,\mathcal{X}_t^1(x))|+|x|^{-2}|d^1(t,\mathcal{X}_t^2(x))-d^2(t,\mathcal{X}_t^2(x))|$ (in the splitting \eqref{split E}, this is $E_t^1+E_t^2$) which controls the error arising through this and whose time derivative can still be controlled through $E$.

The proof then proceeds by considering a weighted $L^2$-norm of these terms (defined in \eqref{def E}), which in particular controls $\norm{|x|^{-1+\frac{\alpha-2}{200}}m(\mathcal{X}_t^1(x),\mathcal{X}_t^2(x))}_{L^2}$ and showing a Gronwall estimate for this energy.

\subsubsection{A possible improvement}

It would of course be very desirable to show a version of this result where the constraint $\alpha>2$ in \eqref{main ass} is replaced by $\alpha\geq 1$, as the assumption is then true for any Lipschitz $w_0$. The author believes that this can be done by taking into account some further, more difficult cancellation effects. More precisely, the limiting factor in the current analysis is the contribution of the transport of $\varpi_t$ near $0$ by the (diverging) velocity $\frac{x^\perp}{2\pi |x|^2}$ to the time derivatives of $d$ and related quantities. However, this contribution should be highly oscillatory since, at leading order, particles are just rotating around the point vortex at a very high frequency and its actual influence should be much smaller than the bounds for the time derivative.

This seems to be much more difficult to implement than the current argument.

\subsection{Structure of the paper}
The rest of the article is devoted to the proof of Theorem \ref{T1}. Section \ref{S2} contains the definition of distance and energy functionals as well as the proof that the Assumption \eqref{main ass} is propagated under the evolution.
Section \ref{S3} contains the proof of the desired Gronwall estimate for the energy.
Section \ref{S4} contains the proofs of some technical lemmas.
An appendix contains some classical estimates.


\section{Preparations and definition of the energy functional}\label{S2}
\subsection{Notation}
We write $A\lesssim B$ if there is a constant $C$ such that $A\leq CB$. Such a constant $C$ is allowed to depend on the initial data $\varpi_0$, the value of the limsup in the assumption \eqref{main ass}, the constants $S_1$ and $S_2$, defined further below and $\alpha$. Such a constant is also allowed to vary from line to line.

We generally work with the variables $\varpi,\bar u,v$ as defined in \eqref{xframe}, where the first equation is understood distributionally.

For the ease of notation, we also use the shorthand $\BS$ for the Biot-Savart law, i.e.\ \begin{align*}
\BS[f]:=\nabla^\perp \Delta^{-1} f=f*\frac{x^\perp}{2\pi|x|^2}.
\end{align*}
For an $x\in \R^2\backslash \{0\}$, we denote its polar angle by \begin{align*}
    x_\theta :=\begin{cases}\cos^{-1}\left(\frac{x_1}{|x|}\right)& \text{for $x_2\geq 0$}\\-\cos^{-1}\left(\frac{x_1}{|x|}\right)& \text{for $x_2< 0$}
    \end{cases}
\end{align*}
with the convention that the range of $\cos^{-1}$ is $[0,\pi]$.
We will always understand these angles (and differences of them) as numbers in $\R/2\pi \Z$. The distance between polar angles is always understood as the distance on $\R/2\pi \Z$, i.e.\ $|x_\theta-y_\theta|:=\min_{k\in \Z}|x_\theta-y_\theta+2\pi k|$.

\subsection{Preliminaries}
Let us first recall that we have \begin{align}
&|\bar u_t(x)|\lesssim 1+|x|^{-1}\label{uni bd u}\\
&|v_t(x)|\lesssim 1\label {bd v}
\end{align}
for all $x$, since the Biot-Savart law maps $L^1\cap L^\infty$ to $L^\infty$ by e.g.\ Young's convolution inequality.

We denote by $\mathcal{X}_t$ the flow map of $u$ defined by \begin{align*}
    \mathcal{X}_0(x)=x,\quad \dfrac \mathcal{X}_t(x)=\bar u_t(\mathcal{X}_t(x)).
\end{align*}
As $\bar u_t$ is locally log-Lipschitz (see Lemma \ref{ll lemma app}), except at $0$, this is well-defined, as long as $X_t(x)\neq 0$, which in fact always is the case for $x\neq 0$. Indeed it follows from the log-Lipschitz property of $v_t$
(see e.g.\ Lemma \ref{ll lemma app}) and the fact that the point vortex velocity is radial around the origin that \begin{align*}
    \left|\dfrac |\mathcal{X}_t(x)|\right|=\left|v_t(\mathcal{X}_t(x))-v_t(0)\right|\lesssim  |\mathcal{X}_t(x)|\log\left(2+\frac{1}{|\mathcal{X}_t(x)|}\right)
\end{align*}
and hence Gronwall's Lemma implies that for $|x|\leq 1$ we have \begin{align}
|\mathcal{X}_t(x)|\gtrsim |x|^{e^{Ct}}\label{bas gronwall}
\end{align}
for some $C$ not depending on $x$ or $t$.

At $0$, where the ODE is not defined, we set $\mathcal{X}_t(0)=0$.

This is a volume-preserving and invertible map and fits into the formalism of regular Lagrangian flows, see e.g.\ \cite{crippa2}.

It is known that every solution $\varpi\in L_t^\infty(L_x^1\cap L_x^\infty)$ of \eqref{xframe} is of the form $\varpi_t=\varpi_0\circ \mathcal{X}_t^{-1}$, see e.g.\ \cite{lacave2009uniqueness}.

\subsection{Propagation of \eqref{main ass} and the compact support of $\varpi$}
Before coming to the actual uniqueness proof, we show that the property \eqref{main ass} is propagated by the evolution and that the support of each solution remains compact.

\begin{lemma}\label{comp supp}
Suppose that the initial datum $\varpi_0\in L^\infty(\R^2)$ is supported in $B_{S_0}(0)$. Then there is a function $S^t:\R_{\geq 0}\rightarrow \R$, only depending on $S_0$ and $\norm{\varpi_0}_{L^\infty}$, such that every solution $\varpi_t$ to the equations \eqref{xframe} is supported in $B_{S_t}(0)$ at time $t$.
\end{lemma}
\begin{proof}
Since $\varpi$ is transported by $u$, we must have that $\varpi_t$ is supported in $B_{S^t}(0)$, where $S^t\geq \max(1, S_0+\int_0^t\norm{\bar u_s}_{L^\infty(\R^2\backslash B_{1}(0))}\dd s)$. Since we have a uniform bound 
$\norm{\bar u_s}_{L^\infty(\R^2\backslash B_{1}(0))}\lesssim 1$ by \eqref{uni bd u},
we see that this $S^t$ can be chosen independently of the solution and finite.
\end{proof}

\begin{lemma}\label{decay prop}
Suppose $\varpi_0$ fulfills \eqref{main ass}, then there is a constant $C$, only depending on \linebreak
$\norm{|\cdot|^{-\alpha}(\varpi_0(\cdot)-\varpi_0(0))}_{L^\infty}$ and $\norm{\varpi_0}_{L^1\cap L^\infty}$, such that every solution $\varpi$ to \eqref{xframe} fulfills \begin{align}
   \norm{|\cdot|^{-\alpha}(\varpi_t(\cdot)-\varpi_t(0))}_{L^\infty(\R^2)}\leq Ce^{Ce^{Ct}} \label{ass later t}
\end{align}
and \begin{align}\label{est v}
\norm{|\cdot|^{-1}(v_t(\cdot)-v_t(0))}_{L^\infty(\R^2)}\leq Ce^{Ct} 
\end{align}
for all $t\in \R_{\geq 0}$.

Furthermore it holds that $\varpi_t(0)=\varpi_0(0)$.
\end{lemma}
\begin{proof}
Since it holds that $\mathcal{X}_t(0)=0$ and $\varpi_t$ is transported, we see that \eqref{main ass} and \eqref{bas gronwall} imply that \begin{align*}
    \norm{\max(1,|x|^{-\alpha e^{-Ct}})(\varpi_t(x)-\varpi_t(0))}_{L_x^\infty}\lesssim 1.
\end{align*} 
%

We then note that \begin{align*}
   |x|^{-1}|v_t(x)-v_t(0)|&\lesssim |x|^{-1}\left|\int \left(\frac{x-z}{|x-z|^2}+\frac{z}{|z|^2}\right) \varpi_t(z)\dd z\right| \\
   &\leq |x|^{-1}\left|\int_{\R^2} \left(\frac{x-z}{|x-z|^2}+\frac{z}{|z|^2}\right)(\varpi_t(z)-\mathds{1}_{B_1}(0)\varpi_t(0))\dd z\right|\\
   &\quad+|x|^{-1}\left|\int_{\R^2} \left(\frac{x-z}{|x-z|^2}+\frac{z}{|z|^2}\right)\mathds{1}_{B_1}(0)\varpi_t(0)\dd z\right|.
\end{align*}
For $|x|\geq \frac{1}{2}$, we trivially have an uniform bound on this because we have a uniform bound on $v$.

For $|x|\leq \frac{1}{2}$, the second term is uniformly bounded because $\frac{|\cdot|}{|\cdot|^2}*\mathds{1}_{B_1(0)}$ is smooth in $B_{\frac{2}{3}}(0)$. The first term is uniformly bounded as \begin{align*}
  \mel  \left|\int_{\R^2} \left(\frac{x-z}{|x-z|^2}+\frac{z}{|z|^2}\right)(\varpi_t(z)-\mathds{1}_{B_1}(0)\varpi_t(0))\dd z\right|\\
  &\lesssim \int_{B_1(0)\cup \supp \varpi_t} |x|\left(|z|^{-2}+|z|^{-1}|x-z|^{-1}\right)\min(1,|z|^{\alpha e^{-Ct}})\dd z\\
    &\lesssim |x|\left(1+e^{Ct}\right)
\end{align*}
because $|\frac{x-z}{|x-z|^2}-\frac{z}{|z|^2}|\lesssim |x|(|z|^{-2}+|z|^{-1}|x-z|^{-1})$ and because the measure of $\supp \varpi_t$ does not depend on $t$.

In sum we have that \begin{align*}|v_t(x)-v_t(0)|\lesssim |x|e^{Ct}\end{align*} with the same constant $C$ for all solutions.

But then we have that the flow map $\mathcal{X}_t$ of every solution fulfils $|\mathcal{X}_t(x)|\geq  e^{-Ce^{Ct}}|x|$ by Gronwall, because the point-vortex velocity is radial around the origin and then we obtain that the property \eqref{ass later t} holds with the same constant for all solutions since $\varpi_t(\cdot)-\varpi_t(0)$ is transported.
\end{proof}

\subsection{Definition of the modulated distance $d$}\label {S24}
Let $\Psi_R$ and $\Psi$ denote the stream functions of $v_t(\cdot)-v_t(0)$ and $\bar u$, both normalized to be $0$ at $0$, i.e.\ 
\begin{align}\label{def psir}
\Psi_R(t,x)=\frac{1}{2\pi}\int_{\R^2} \left(\log|x-z|-\log|z|+\frac{x\cdot z}{|z|^2}\right) \varpi_t(z)\dd z
\end{align}
and
\begin{align*}
\Psi(t,x)=\frac{1}{2\pi}\log|x|+\Psi_R(t,x).
\end{align*}
We set  
\begin{align*}
d(t,x):=\exp(2\pi \Psi(t,x)).
\end{align*}
Note that \begin{align}
\bar u_t(x)\perp \nabla_x d(t,x)\label{ort}
\end{align}
for all $x$ and $t$ as one sees from the fact that $\bar{u}_t=\nabla_x^\perp \Psi$ and the chain rule.

Since $\varpi_t\in L^\infty$, we have \begin{align}
\norm{\Psi_R(t,\cdot)}_{L^\infty}\lesssim 1,\label{linfty psir}
\end{align}
uniformly in $t$.

%

These quantities all depend on $\varpi_t$. Below in the uniqueness proof, we will compare two solutions $\varpi^1$ and $\varpi^2$, to denote the dependence on the solution, we will use a superscript as in e.g.\ $d^1(t,x)$ and $d^2(t,x)$. We remark that this kind of modulated distance was already used in the author's earlier work \cite{meyer2025long} to study the long-time confinement of concentrated vortices.

Let us note the useful identities \begin{align}
    &\dfrac d(t,\mathcal{X}_t(x))=(\de_t d)(t,\mathcal{X}_t(x))\label{id flow}\\
    &d(t,x)=|x|e^{2\pi\Psi_R(t,x)}\label{psir id}
\end{align}
for future reference. \eqref{id flow} follows from \eqref{ort} and the chain rule.

At leading order, $d$ behaves like the distance $|x|$, but varies very slowly with time, the following lemma gives a precise account of this relationship:
\begin{lemma}
Suppose $\varpi$ is a solution to \eqref{xframe} on $[0,T]$, supported on $B_S(0)$, which fulfills \eqref{ass later t}.
Then\begin{align}
    &\left||x|-d(t,x)\right|\lesssim |x|^3 \label{d est 1}\\
    &\left|\nabla_x|x|-\nabla_x d(t,x)\right|\lesssim |x|^2\label{d est 2}\quad \text{for $x\neq 0$}
\end{align}
and \begin{align}
  \left|\dfrac d(t,x)\right|\lesssim |x|^3\label{d est 3},
\end{align}
for all $x\in B_{10S}(0)$ and $t\in [0,T]$, where the implicit constant depends on $S$ and $T$.
\end{lemma}
In particular, it holds that \begin{align}
    d(t,x)\approx |x|.\label{d est 4}
\end{align}
for all sufficiently small $|x|$.
\begin{proof}
For \eqref{d est 1}, we use that $\Psi_R$ has a double zero at $0$, indeed, as already argued above in \eqref{est v}, its orthogonal derivative $v_t(\cdot)-v_t(0)$ has a first order zero at $0$. Therefore \eqref{psir id} implies \eqref{d est 1}.

Similarly, \eqref{d est 2} follows from the product rule and because $|\nabla_x\Psi_R(t,x)|=|v_t(x)-v_t(0)|\lesssim |x|$, by \eqref{est v}.

For \eqref{d est 3}, we use that $|\dfrac \Psi_R(t,x)|\lesssim |x|^2$, which is proven further below as part of Lemma \ref{det d} and which yields the statement by \eqref{psir id}.
\end{proof}

\subsection{Definition of the distance functional $D$}

To define the distance functional, we fix an initial datum $\varpi_0$ which fulfils the assumptions of the theorem.

Clearly, it suffices to show uniqueness on a short time interval, the global-in-time result then follows by applying the short-time result to the initial datum $\varpi_t$, which still fulfils the same assumptions by the Lemmata \ref{comp supp} and \ref{decay prop}.

Let us take two solutions $\varpi^1$ and $\varpi^2$ of \eqref{xframe} for this initial datum of which we wish to show that they are the same.

Let us denote by $\bar u_t^1,v_t^1,\bar u_t^2,v_t^2$ the corresponding velocities and let $\mathcal{X}_t^1$ and $\mathcal{X}_t^2$ be the corresponding flow maps. 

Let us take an $S_1\in (0,\frac{1}{10})$ such that for $x\in B_{S_1}(0)$ it holds that \begin{align}
    \frac{99}{100}|x|\leq |\mathcal{X}_t^1(x)|,|\mathcal{X}_t^2(x)|\leq \frac{101}{100}|x|\label{mod}
\end{align}
for all $t\in [0,1]$. This is possible by e.g.\ \eqref{id flow}, \eqref{d est 3} and \eqref{d est 4}.

This implies that \begin{align}
  \frac{12}{10}|x|\geq   \frac{11}{10}\min(|\mathcal{X}_t^1(x)|,|\mathcal{X}_t^2(x)|)\geq \max(|\mathcal{X}_t^1(x)|,|\mathcal{X}_t^2(x)|)\geq \frac{9}{10}|x|\label{mod1}
\end{align}
for all $t$ on some small time interval $[0,T]$ (where $T$ does not depend on $x$) and all $x\in \R^2$. Indeed for small $|x|$ it follows from \eqref{mod}, while for $|x|\geq S_1$
it holds because $u$ is uniformly bounded outside of $B_{S_1}(0)$.

After potentially lowering $S_1$, we can also assume by \eqref{d est 2} and \eqref{d est 4} that \begin{align}
    d^1(t,x),d^2(t,x)\in [\frac{1}{2}|x|,2|x|], \quad \left|\nabla_x d^1(t,x)-\nabla |x|\right|+\left|\nabla_x d^2(t,x)-\nabla |x|\right|\leq \frac{1}{10} \qquad \text{for $x\in B_{S_1}(0)$}\label{d mod S1}
\end{align}
and $t\in [0,T]$.

We then consider the following regions in $\{(x,y)\in \R^{2+2}\,|\,\frac{11}{10}\min(|x|,|y|)\geq \max(|x|,|y|)\}$: \begin{align*}
&\Omega:=\left\{(x,y)\,\Big|\, \max(|x|,|y|)\leq \frac{1}{10}S_1,\, \frac{11}{10}\min(|x|,|y|)\geq \max(|x|,|y|)\right\}\\
&\Omega':=\left\{(x,y)\,\Big|\,\max(|x|,|y|)\leq \frac{1}{5}S_1,\, \frac{11}{10}\min(|x|,|y|)\geq \max(|x|,|y|)\right\}.
\end{align*}
We take a smooth, nonnegative cut-off function $\eta$ with values in $[0,1]$, which is supported in $\Omega'$ and which equals $1$ on $\Omega$.

Let us fix an $S_2\geq 100$ such that $\supp \varpi_t^1,\varpi_t^2\subset B_{\frac{1}{2}S_2}(0)$ for all $t\in [0,T]$, this is possible by Lemma \ref{comp supp}.

Let us set \begin{align*}
    \beta:=\frac{\alpha-2}{100}>0.
\end{align*}
We then define \begin{align*}\begin{aligned}
\mel D(x,y)^2:= \eta(x,y)\biggl(\frac{1}{d^1(t,x)^{4-\beta}}|\Psi_R^1(t,x)-\Psi_R^2(t,x)|^2+\frac{1}{d^2(t,y)^{4-\beta}}|\Psi_R^1(t,y)-\Psi_R^2(t,y)|^2\\
&\quad+\frac{1}{\max(d^1(t,x),\,d^2(t,y))^{6-\beta}}|d^1(t,x)-d^2(t,y)|^2+\max(d^1(t,x),\,d^2(t,y))^{\beta}|x_\theta-y_\theta|^2\biggr)\\
&\quad+(1-\eta(x,y))|x-y|^2
\end{aligned}\end{align*}
and take $D$ itself as the non-negative square root of this.

The following lemma shows that this indeed controls the distance $|x|^{\beta/2-1}m(x,y)$.

\begin{lemma}
Whenever $\frac{11}{10}\min(|x|,|y|)\geq \max(|x|,|y|)$, we have  \begin{align}\label{D m}
D(x,y)\gtrsim |x|^{\frac{\beta}{2}-1}m(x,y)\gtrsim |x|^{\frac{\beta}{2}-1}|x-y| \quad \text{ in $B_{2S_2}(0)$}.
\end{align}
In particular, $D$ is non-negative and if $D(x,y)=0$, then $x=y$.
\end{lemma}

\begin{proof}
Recall the definition \eqref{def m} of $m$.
Let us first note that $m(x,y)\gtrsim |x-y|$ on any compact set and that the statement is hence trivial outside of $\Omega'$. It therefore suffices to show that \begin{align}\begin{aligned}
 \mel   \frac{1}{d^1(t,x)^{4-\beta}}|\Psi_R^1(t,x)-\Psi_R^2(t,x)|^2+\frac{1}{d^2(t,y)^{4-\beta}}|\Psi_R^1(t,y)-\Psi_R^2(t,y)|^2\\
&\quad+\frac{1}{\max(d^1(t,x),\,d^2(t,y))^{6-\beta}}|d^1(t,x)-d^2(t,y)|^2+\max(d^1(t,x),\,d^2(t,y))^{\beta}|x_\theta-y_\theta|^2\\
&\gtrsim |x|^{\frac{\beta}{2}-1}m(x,y)\label{nec est}\end{aligned}\end{align}
in $\Omega'$.

We directly control the difference of the polar angles by \eqref{d mod S1}.

To control the difference of the moduli, we note that by \eqref{psir id} and \eqref{d mod S1}, it holds that \begin{align*}
\mel\frac{1}{d^1(t,x)^{2-\frac{\beta}{2}}}\left|\Psi_R^1(t,x)-\Psi_R^2(t,x)\right|\gtrsim |x|^{-2+\frac{\beta}{2}}\left|e^{2\pi\Psi_R^1(t,x)}-e^{2\pi\Psi_R^2(t,x)}\right|\\
&=|x|^{-3+\frac{\beta}{2}}\left|d^1(t,x)-d^2(t,x)\right|,
\end{align*}
since $\Psi_R^1$ and $\Psi_R^2$ are uniformly bounded by \eqref{linfty psir}.

Therefore, using \eqref{d mod S1} again, we see that 
\begin{align}\begin{aligned}
\mel\frac{1}{\max(d^1(t,x),\,d^2(t,y))^{3-\frac{\beta}{2}}}|d^1(t,x)-d^2(t,y)|+\frac{1}{d^1(t,x)^{2-\frac{\beta}{2}}}|\Psi_R^1(t,x)-\Psi_R^2(t,x)|\\
&\gtrsim |x|^{-3+\frac{\beta}{2}}|d^2(t,x)-d^2(t,y)|.\label{d1est}
\end{aligned}\end{align}
We now claim that for $(x,y)\in \Omega'$, it holds that \begin{align}
\frac{1}{|x|^2}|d^2(t,x)-d^2(t,y)|+|x||x_\theta-y_\theta| \gtrsim  |x|^{-2}||x|-|y||\label{claim 1}.
\end{align}
To see this, we let $y'$ be the unique point such that $y_\theta'=y_\theta$ and $|y'|=|x|$. Then we conclude from the fact that $|z^\perp\cdot\nabla d^2(t,z)|\lesssim |x|^{3}$ on the circle segment connecting $y'$ and $x$ by \eqref{d est 2}, we have that \begin{align}
|d^2(t,x)-d^2(t,y')|\lesssim |x|^3|x_\theta-y_\theta'|,\label{yprime1}
\end{align}
since the circle segment has length $\approx |x||x_\theta-y_\theta'|$.

On the other hand, we also have $z\cdot \nabla_z d^2(t,z)\approx |z|$ for sufficiently small $z$ by \eqref{d mod S1} and hence it holds that   \begin{align}
    |d^2(t,y)-d^2(t,y')|\gtrsim ||y|-|y'||=||y|-|x||.\label{yprime2}
\end{align}
Combining \eqref{yprime1} and \eqref{yprime2} and using the triangle inequality yields the claim \eqref{claim 1}.

We hence see from \eqref{d1est} and \eqref{claim 1}, that for $|x|\leq S_1$, we have that \begin{align*}
    D(x,y)\gtrsim |x|^{\frac{\beta}{2}-3}\left||x|-|y|\right|,
\end{align*}
showing the remaining part of \eqref{nec est}.
\end{proof}

\subsection{The energy functional $E_t$}
Our energy functional for the proof of uniqueness is  \begin{align}\label{def E}
E_t:=\int_{B_{S_2}(0)} D(\mathcal{X}_t^1(x),\mathcal{X}_t^2(x))^2\dd x.
\end{align}
We observe the following: \begin{itemize}
\item $E_0=0$, since $D(x,x)=0$ if $\Psi_R^1=\Psi_R^2$.
\item If $E_t=0$, then both flow maps and hence also $\varpi_t^1$ and $\varpi_t^2$ agree.
\item $E_t$ can be infinite if $\mathcal{X}_t^1$ and $\mathcal{X}_t^2$ are diffeomorphisms that do not come from a solution to \eqref{xframe}. If they do come from two solutions to \eqref{xframe} with the same initial datum, then Lemma \ref{e fin} below shows that $E_t$ is finite for all $t$.
\end{itemize}

To prove uniqueness, it suffices to show that \begin{align}
\dfrac E_t\lesssim E_t \log\left(2+\frac{1}{E_t}\right)  \quad \text{ for $t\in [0,T]$}.\label{dest}
\end{align}
Gronwall's lemma then implies $E_t=0$ for $t\in [0,T]$, showing that $\varpi_t^1=\varpi_t^2$ and that uniqueness holds.

As already explained above, uniqueness after the time $T$ follows from applying this short time statement to the initial datum $\varpi_T$ (with different constants $S_1$ and $S_2$) since all the assumptions on the initial data are propagated. More precisely, this shows that the time interval on which $\varpi^1=\varpi^2$ is an open, nonempty subset of $\R_{\geq 0}$. Since the solutions are also continuous in time in e.g.\ $W^{-1,1}$, as the time derivatives are uniformly bounded in that space, the time interval on which $\varpi^1=\varpi^2$ is also closed. Since $\R_{\geq 0}$ is connected, this time interval must be the full $\R_{\geq 0}$.

To show this, we split $E_t=\sum_{i=1}^5E_t^i$ as \begin{align}
\begin{aligned}
&E_t^1:=\int_{B_{S_2}(0)}\eta(\mathcal{X}_t^1(x),\mathcal{X}_t^2(x))\frac{1}{d^1(t,\mathcal{X}_t^1(x))^{4-\beta}}|\Psi_R^1(t,\mathcal{X}_t^1(x))-\Psi_R^2(t,\mathcal{X}_t^1(x))|^2\dx\\
&E_t^2:=\int_{B_{S_2}(0)}\eta(\mathcal{X}_t^1(x),\mathcal{X}_t^2(x))\frac{1}{d^2(t,\mathcal{X}_t^2(x))^{4-\beta}}|\Psi_R^1(t,\mathcal{X}_t^2(x))-\Psi_R^2(t,\mathcal{X}_t^2(x))|^2\dx\\
&E_t^3:=\int_{B_{S_2}(0)}\frac{\eta(\mathcal{X}_t^1(x),\mathcal{X}_t^2(x))}{\max(d^1(t,\mathcal{X}_t^1(x)),\,d^2(t,\mathcal{X}_t^2(x)))^{6-\beta}}|d^1(t,\mathcal{X}_t^1(x))-d^2(t,\mathcal{X}_t^2(x))|^2\dx\\
&E_t^4:=\int_{B_{S_2}(0)}\eta(\mathcal{X}_t^1(x),\mathcal{X}_t^2(x))\max(d^1(t,\mathcal{X}_t^1(x)),\,d^2(t,\mathcal{X}_t^2(x)))^\beta|\mathcal{X}_t^1(x)_\theta-\mathcal{X}_t^2(x)_\theta|^2\dd x\\
&E_t^5:=\int_{B_{S_2}(0)}(1-\eta)(\mathcal{X}_t^1(x),\mathcal{X}_t^2(x))|\mathcal{X}_t^1(x)-\mathcal{X}_t^2(x)|^2\dd x,\label{split E}
\end{aligned}\end{align}
and estimate the time derivative of each summand separately.

\section{Proof of the differential estimate \eqref{dest} for $E_t$}\label{S3}

\subsection{Preparations and Notation}

To avoid line breaks, we introduce a few additional shorthands.

We set \begin{align*}
 &   k_t(x):=m(\mathcal{X}_t^1(x),\mathcal{X}_t^2(x))\\
&    \mathfrak{d}_t(x):=\max\left(d^1(t,\mathcal{X}_t^1(x)),d^2(t,\mathcal{X}_t^2(x))\right).
\end{align*}
Note that \eqref{d mod S1} and \eqref{mod} imply that for $|x|\leq \frac{1}{2}S_1$, it holds that \begin{align}
    \mathfrak{d}_t(x)\approx |x|. \label{mfrakd}
\end{align}

\begin{lemma}\label{e fin}
For $t\in [0,T]$, it holds that \begin{align*}
D(\mathcal{X}_t^1(x),\mathcal{X}_t^2(x))^2+E_t<\infty
\end{align*}
for all $x\in B_{S_2}(0)$.
\end{lemma}


\begin{proof}
Obviously, finiteness of $D$ implies finiteness of $E_t$, so it suffices to argue the former. 

The finiteness of the summand corresponding to $E_t^5$ is a trivial consequence of the boundedness of $\bar u$.

It follows from \eqref{psir id}, \eqref{d est 1} and \eqref{mfrakd} that \begin{align*}
    \mel\frac{1}{d^1(t,x)^{4-\beta}}\left|\Psi_R^1(t,\mathcal{X}_t^1(x))-\Psi_R^2(t,\mathcal{X}_t^1(x))\right|^2\\
    &\lesssim |x|^{-6+\beta} \left(\left||\mathcal{X}_t^1(x)|-d^1(t,\mathcal{X}_t^1(x))\right|^2+\left||\mathcal{X}_t^1(x)|-d^2(t,\mathcal{X}_t^1(x))\right|^2\right)\lesssim 
    |x|^{\beta},
\end{align*}    
  while it follows from \eqref{id flow}, \eqref{d est 1}, \eqref{d est 3} and \eqref{mfrakd} that \begin{align*}
  \mel\left|d^1(t,\mathcal{X}_t^1(x))-d^2(t,\mathcal{X}_t^2(x))\right|\lesssim \left|d^1(t,\mathcal{X}_t^1(x))-|x|\right|+\left|d^2(t,\mathcal{X}_t^2(x))-|x|\right|\\
  &\lesssim\left|d^1(0,x)-|x|\right|+\int_0^t \left|(\de_t d^1)(s,\mathcal{X}_s^1(x))\right|+\left|(\de_t d^2)(s,\mathcal{X}_s^2(x))\right|\dd s\lesssim |x|^3,\end{align*}
which also used that $d^1(0,\cdot)=d^2(0,\cdot)$.
By \eqref{mfrakd}, this implies that   \begin{align*}
    \frac{1}{\max(d^1(t,\mathcal{X}_t^1(x)),\,d^2(t,\mathcal{X}_t^2(x)))^{6-\beta}}\left|d^1(t,\mathcal{X}_t^1(x))-d^2(t,\mathcal{X}_t^2(x))\right|^2\lesssim |x|^{\beta}.
\end{align*}
The finiteness of the angular term is trivial. 
\end{proof}

\begin{lemma}\label{ll lemma}
For $t\in [0,T]$ and $\max(|x|,|y|)\leq 2\min(|x|,|y|)\leq 2S_2$, we have \begin{align}
  &  \left|\bar u_t(x)-\bar u_t(y)\right|\lesssim  |x|^{-2}m(x,y)\log\left(2+ \frac{1}{m(x,y)}\right)\label{ll1}\\
  & \left|\bar u_t(x)\cdot \frac{x^\perp}{|x|^2}-\bar u_t(y)\cdot \frac{y^\perp}{|y|^2}\right|\lesssim  |x|^{-1}m(x,y)\log\left(2+\frac{1}{m(x,y)}\right)\label{ll2},
\end{align}
in the sense that these estimates hold for both $\bar u_t^1$ and $\bar u_t^2$ in place of $\bar u_t$.
\end{lemma}


%

\begin{proof}
We have \begin{align*}
\left|v_t(x)-v_t(y)\right|\lesssim |x-y|\log\left(2+\frac{1}{|x-y|}\right)
\end{align*}
by Lemma \ref{ll lemma app}. Since $|x-y|\lesssim m(x,y)$ on compact sets by definition, we can estimate this part with \begin{align*}
 \left|v_t(x)-v_t(y)\right|\lesssim   m(x,y)\log\left(2+ \frac{1}{m(x,y)}\right).
\end{align*}    
    

\noindent For the point-vortex velocity, we note that \begin{align*}
\left|\frac{x^\perp}{|x|^2}-\frac{y^\perp}{|y|^2}\right|\leq \frac{|x-y|}{|x|^2}+\frac{|y|||x|^2-|y|^2|}{|x|^2|y|^2}\lesssim \frac{|x-y|}{|x|^2}
\end{align*}
by the assumption that $|x|$ and $|y|$ are of the same scale.
This has the desired bound \eqref{ll1}.


For the estimate \eqref{ll2}, we note that \begin{align*}
\left|(v_t(x)-v_t(0))\cdot \frac{x^\perp}{|x|}-(v_t(y)-v_t(0))\cdot \frac{y^\perp}{|y|}\right|\leq \left|\frac{x}{|x|^2}-\frac{y}{|y|^2}\right||v_t(y)-v_t(0)|+|x|^{-1}|v_t(x)-v_t(y)|.
\end{align*}
The first difference is estimated as \begin{align*}
  \left|\frac{x}{|x|^2}-\frac{y}{|y|^2}\right||v_t(y)-v_t(0)|\lesssim |x|^{-1}|x-y|\lesssim |x|^{-1}m(x,y)  
\end{align*}
with \eqref{est v}, and the fact that $\left|\frac{x}{|x|^2}-\frac{y}{|y|^2}\right|\lesssim |x|^{-2}|x-y|$.

The other term is estimated as above.

Regarding the contribution of the point vortex, we note that \begin{align*}
\left|\frac{x^\perp}{|x|^2}\cdot \frac{x^\perp}{|x|^2}-\frac{y^\perp}{|y|^2}\cdot \frac{y^\perp}{|y|^2}\right|=\left|\frac{1}{|x|^2}-\frac{1}{|y|^2}\right|\lesssim |x|^{-3}||x|-|y||\lesssim |x|^{-1}m(x,y).
\end{align*}
\end{proof}

\subsection{A priori bounds on $d$, $\Psi_R$ and $v_t$}
In this subsection, we collect some bounds on the time derivatives of the $d$ and $\Psi_R$, which will be used in the proof of the differential estimate on $E_t$ below. All of their proofs are in principle straightforward, but somewhat tedious and therefore postponed to Section \ref{S4} further below.
 
\begin{lemma}\label{deri dpsi}
It holds that \begin{align}
&\int_{B_{\frac{1}{2}}(0)}|\de_t\Psi_R^1(t,x)-\de_t\Psi_R^2(t,x)|^2|x|^{\beta-4}\dd x\lesssim E_t\left(\log\left(2+\frac{1}{E_t}\right)\right)^2\label{deri dpsi1}\\
&\int_{B_{\frac{1}{2}}(0)}|\Psi_R^1(t,x)-\Psi_R^2(t,x)|^2|x|^{\beta-4}\dd x\lesssim E_t.\label{deri dpsi2}
\end{align}
\end{lemma}

\begin{lemma}\label{det d}
Suppose that $\max(|x|,|y|)\leq 2\min(|x|,|y|)\leq S_1$, then for all $t\in [0,T]$ it holds that \begin{align}
&|\de_t \Psi_R^1(t,x)|\lesssim |x|^2\label{det d 1}\\
&|\Psi_R^1(t,x)-\Psi_R^1(t,y)|\lesssim |x||x-y|\label{det d 2}\\
&|\de_t \Psi_R^1(t,x)-\de_t \Psi_R^1(t,y)|\lesssim |x||x-y|\log\left(2+\frac{1}{|x-y|}\right)\label{det d 3},
\end{align}
in particular
\begin{align}
    \left|\de_t d^1(t,x)-\de_t d^1(t,y)\right|\lesssim |x|^2|x-y|\log\left(2+\frac{1}{|x-y|}\right)\label{det d 4}
\end{align}
and the exact same estimates hold for $\Psi_R^2(t,\cdot)$ and $d^2(t,\cdot)$.
\end{lemma}

\begin{lemma}\label{weighted l2}
It holds that \begin{align}\label{weighted l2 est}
\int_{B_1(0)} |x|^{\beta-2}\left|v_t^1(x)-v_t^2(x)\right|^2\dx\lesssim E_t
\end{align}
and \begin{align}
    \left|v_t^1(0)-v_t^2(0)\right|\lesssim \sqrt{E_t}.\label{vt 0}
\end{align}
\end{lemma}

\subsection{Bound on $E_t^1$ and $E_t^2$}
By the incompressibility of the flow, we can rewrite \begin{align*}
E_t^1=\int_{B_\frac{1}{2}(0)}\eta(x,x)|\Psi_R^1(t,x)-\Psi_R^2(t,x)|^2d^1(t,x)^{\beta-4}\dd x,
\end{align*}
which also used that $(\mathcal{X}_t^1)^{-1}$ maps $\supp \eta$ into $B_\frac{1}{2}(0)^2$ by \eqref{mod1}.

We then estimate \begin{align*}
\mel\left|\dfrac\int_{B_\frac{1}{2}(0)}\eta(x,x)|\Psi_R^1(t,x)-\Psi_R^2(t,x)|^2d^1(t,x)^{\beta-4}\dd x\right|\\
&\lesssim 
\int_{B_\frac{1}{2}(0)}\eta(x,x)\biggl(\frac{|\de_t d^1(t,x)|}{d^1(t,x)}|\Psi_R^1(t,x)-\Psi_R^2(t,x)|^2\\
&\qquad+\left|\de_t\Psi_R^1(t,x)-\de_t\Psi_R^2(t,x)\right|\left|\Psi_R^1(t,x)-\Psi_R^2(t,x)\right|\biggr)d^1(t,x)^{\beta-4}\dd x.
\end{align*}
The first summand is $\lesssim E_t$ by \eqref{d est 3} and \eqref{d mod S1}.

The second summand is $\lesssim E_t\log(2+\frac{1}{E_t})$ by Lemma \ref{deri dpsi} above and Cauchy-Schwarz, showing the desired bound on this part of the time derivative.

Clearly $E_t^2$ can be estimated in the same way.

\subsection{Bound on $E_t^3$}

We compute \begin{align}
\begin{aligned}\label{deri 1}
&\dfrac E_t^3=\dfrac \int_{B_{S_2}(0)} \eta(\mathcal{X}_t^1(x),\mathcal{X}_t^2(x))\left|d^1(t,\mathcal{X}_t^1(x))-d^2(t,\mathcal{X}_t^2(x))\right|^2\mathfrak{d}_t(x)^{\beta-6}\dd x\\
&\lesssim \int_{B_{S_2}(0)} \left|\nabla \eta(\mathcal{X}_t^1(x),\mathcal{X}_t^2(x))\cdot\binom{\bar u_t^1(\mathcal{X}_t^1(x))}{\bar u_t^2(\mathcal{X}_t^2(x))}\right|\left|d^1(t,\mathcal{X}_t^1(x))-d^2(t,\mathcal{X}_t^2(x))\right|^2\mathfrak{d}_t(x)^{\beta-6}\\
&\:  +\eta(\mathcal{X}_t^1(x),\mathcal{X}_t^2(x))\left|d^1(t,\mathcal{X}_t^1(x))-d^2(t,\mathcal{X}_t^2(x))\right|^2\mathfrak{d}_t(x)^{\beta-6} \frac{\max(|(\de_t d^1)(t,\mathcal{X}_t^1(x))|,|(\de_t d^2)(t,\mathcal{X}_t^2(x))|)}{\mathfrak{d}_t(x)}\\
&\:+\eta(\mathcal{X}_t^1(x),\mathcal{X}_t^2(x))\left|(\de_t d^1)(t,\mathcal{X}_t^1(x))-(\de_t d^2)(t,\mathcal{X}_t^2(x))\right|\left|d^1(t,\mathcal{X}_t^1(x))-d^2(t,\mathcal{X}_t^2(x))\right|\mathfrak{d}_t(x)^{\beta-6}\dd x,
\end{aligned}\end{align}
which crucially uses the formula \eqref{id flow}.


For the first term, we use that $|\bar u_t^1|+|\bar u_t^2|+|d^1(t,\cdot)|^{-1}+|d^2(t,\cdot)|^{-1}\lesssim 1$ on $\supp\nabla\eta$ by \eqref{uni bd u} and \eqref{d mod S1} and that the gradient of the cutoff function is only non-zero for $x\in B_{\frac{1}{2}S_1}(0)\backslash B_{\frac{1}{20}S_1}(0)$  by \eqref{mod} to estimate \begin{align*}
  \mel  \left|\nabla \eta(\mathcal{X}_t^1(x),\mathcal{X}_t^2(x))\cdot\binom{\bar u_t^1(\mathcal{X}_t^1(x))}{\bar u_t^2(\mathcal{X}_t^2(x))}\right|\left|d^1(t,\mathcal{X}_t^1(x))-d^2(t,\mathcal{X}_t^2(x))\right|^2\mathfrak{d}_t(x)^{\beta-6}\\
  &\lesssim \mathds{1}_{B_{\frac{1}{2}S_1}(0)\backslash B_{\frac{1}{20}S_1}(0)}(x) \left|d^1(t,\mathcal{X}_t^1(x))-d^2(t,\mathcal{X}_t^2(x))\right|^2\mathfrak{d}_t(x)^{\beta-6}\\
  &\lesssim \mathds{1}_{B_{\frac{1}{2}S_1}(0)\backslash B_{\frac{1}{20}S_1}(0)}(x)|x|^{\beta-6}\Bigl(\left|d^2(t,\mathcal{X}_t^1(x))-d^2(t,\mathcal{X}_t^2(x))\right|^2\\
  &\quad+|x|^2\left|e^{2\pi\Psi_R^1(t,\mathcal{X}_t^1(x))}-e^{2\pi\Psi_R^2(t,\mathcal{X}_t^1(x))}\right|^2\Bigr)\\
 & \lesssim \mathds{1}_{B_{\frac{1}{2}S_1}(0)\backslash B_{\frac{1}{20}S_1}(0)}(x)\left(|\mathcal{X}_t^1(x)-\mathcal{X}_t^2(x)|^2+|\Psi_R^1(t,\mathcal{X}_t^1(x))-\Psi_R^2(t,\mathcal{X}_t^1(x))|^2\right)
\end{align*}
where we have also used \eqref{mfrakd}, \eqref{linfty psir}, the identity \eqref{psir id} and that $d^2(t,\cdot)$ is uniformly $W^{1,\infty}$ by \eqref{d est 2}.

Applying \eqref{D m}, \eqref{deri dpsi2} and using that all $x$-dependent weights are $\approx 1$ on $B_{\frac{1}{2}S_1}(0)\backslash B_{\frac{1}{20}S_1}(0)$, we therefore see that \begin{align*}
\int_{B_{S_2}(0)} \left|\nabla \eta(\mathcal{X}_t^1(x),\mathcal{X}_t^2(x))\cdot\binom{\bar u_t^1(\mathcal{X}_t^1(x))}{\bar u_t^2(\mathcal{X}_t^2(x))}\right|\left|d^1(t,\mathcal{X}_t^1(x))-d^2(t,\mathcal{X}_t^2(x))\right|^2\mathfrak{d}_t(x)^{\beta-6}\dx\lesssim E_t.
\end{align*}

\noindent The second term in \eqref{deri 1} is trivially bounded by $E_t$ thanks to the estimates on $d^1$ and $d^2$ in \eqref{d est 3} and \eqref{d mod S1}.

The third term on the other hand is estimated as \begin{align}
\begin{aligned}\label{3e3}
\mel \int_{B_{S_2}(0)} \eta(\mathcal{X}_t^1(x),\mathcal{X}_t^2(x))\left| d^1(t,\mathcal{X}_t^1(x))-d^2(t,\mathcal{X}_t^2(x))\right|\left|(\de_t d^1)(t,\mathcal{X}_t^1(x))-(\de_t d^2)(t,\mathcal{X}_t^2(x))\right|\mathfrak{d}_t(x)^{\beta-6}\dd x\\
&\lesssim \int_{B_{S_2}(0)} \eta(\mathcal{X}_t^1(x),\mathcal{X}_t^2(x))\left| d^1(t,\mathcal{X}_t^1(x))-d^2(t,\mathcal{X}_t^2(x))\right|\Bigl(\left|(\de_t d^1)(t,\mathcal{X}_t^1(x))-(\de_t d^1)(t,\mathcal{X}_t^2(x))\right|\\
&\qquad +\left|(\de_t d^1)(t,\mathcal{X}_t^2(x))-(\de_t d^2)(t,\mathcal{X}_t^2(x))\right|\Bigr)|x|^{\beta-6}\dd x,
\end{aligned}\end{align}
where we used \eqref{mfrakd}.

The first summand in this is directly estimated by Lemma \ref{det d}  as \begin{align*}
\mel\int_{B_{S_2}(0)} \eta(\mathcal{X}_t^1(x),\mathcal{X}_t^2(x))\left| d^1(t,\mathcal{X}_t^1(x))-d^2(t,\mathcal{X}_t^2(x))\right|\left|(\de_t d^1)(t,\mathcal{X}_t^1(x))-(\de_t d^1)(t,\mathcal{X}_t^2(x))\right||x|^{\beta-6}\dd x\\
&\lesssim \int \eta(\mathcal{X}_t^1(x),\mathcal{X}_t^2(x))\left| d^1(t,\mathcal{X}_t^1(x))-d^2(t,\mathcal{X}_t^2(x))\right|\left|\mathcal{X}_t^1(x)-\mathcal{X}_t^2(x)\right|\\
&\qquad \times \log\left(2+\frac{1}{\left|\mathcal{X}_t^1(x)-\mathcal{X}_t^2(x)\right|}\right)|x|^{\beta-4}\dd x\\
&\lesssim \sqrt{E_t}\left(\int_{B_{S_2}(0)}\eta(\mathcal{X}_t^1(x),\mathcal{X}_t^2(x)) \left|\mathcal{X}_t^1(x)-\mathcal{X}_t^2(x)\right|^2\log^2\left(2+\frac{1}{\left|\mathcal{X}_t^1(x)-\mathcal{X}_t^2(x)\right|}\right)|x|^{\beta-2}\dd x\right)^\frac{1}{2}
\end{align*}
where the last step uses Cauchy-Schwarz. 

This can be further estimated with Jensen's inequality, applied to the concave function $y\log(2+y^{-\frac{1}{2}})^2$ as \begin{align*}
  \mel\int_{B_{S_2}(0)}\eta(\mathcal{X}_t^1(x),\mathcal{X}_t^2(x)) \left|\mathcal{X}_t^1(x)-\mathcal{X}_t^2(x)\right|^2\log^2\left(2+\frac{1}{\left|\mathcal{X}_t^1(x)-\mathcal{X}_t^2(x)\right|}\right)|x|^{\beta-2}\dd x\\
  &\lesssim \int_{B_{S_2}(0)} \eta(\mathcal{X}_t^1(x),\mathcal{X}_t^2(x)) \left|\mathcal{X}_t^1(x)-\mathcal{X}_t^2(x)\right|^2|x|^{\beta-2}\dd x\\
&\quad\times\log\left(2+\left(A^{-1}\int_{B_{S_2}(0)} \eta(\mathcal{X}_t^1(x),\mathcal{X}_t^2(x)) \left|\mathcal{X}_t^1(x)-\mathcal{X}_t^2(x)\right|^2|x|^{\beta-2}\dd x\right)^{-\frac{1}{2}}\right)^2\\
&\lesssim E_t\log\left(2+\frac{1}{E_t}\right)^2,
\end{align*}
where $A=\int_{B_{S_2}(0)}\eta(\mathcal{X}_t^1(x),\mathcal{X}_t^2(x))|x|^{\beta-2}\dd x\approx 1$ is renormalisation factor to ensure that \newline $\mathds{1}_{B_{S_2}(0)}\eta(\mathcal{X}_t^1(x),\mathcal{X}_t^2(x))|x|^{\beta-2}$ is a probability measure and the last step uses \eqref{D m}.

%
For the other summand in \eqref{3e3}, we recall the formula \eqref{psir id}, and see that \begin{align*}
&|(\de_t d^1)(t,\mathcal{X}_t^2(x))-(\de_t d^2)(t,\mathcal{X}_t^2(x))|\\
&\quad=2\pi\left|\mathcal{X}_t^2(x)\right|\Bigl|e^{2\pi\Psi_R^1(t,\mathcal{X}_t^2(x))}(\de_t \Psi_R^1)(t,\mathcal{X}_t^2(x))-e^{2\pi\Psi_R^2(t,\mathcal{X}_t^2(x))}(\de_t\Psi_R^2)(t,\mathcal{X}_t^2(x))\Bigr|\\
&\quad\lesssim |x|\Bigl(\left|\Psi_R^1(t,\mathcal{X}_t^2(x))-\Psi_R^2(t,\mathcal{X}_t^2(x))\right|\left|(\de_t\Psi_R^1)(t,\mathcal{X}_t^2(x))\right|+\left|(\de_t\Psi_R^1)(t,\mathcal{X}_t^2(x))-(\de_t\Psi_R^2)(t,\mathcal{X}_t^2(x))\right|\Bigr)\\
&\quad\lesssim|x|\Bigl(\left|d^1(t,\mathcal{X}_t^2(x))-d^2(t,\mathcal{X}_t^2(x))\right||x|+\left|(\de_t\Psi_R^1)(t,\mathcal{X}_t^2(x))-(\de_t\Psi_R^2)(t,\mathcal{X}_t^2(x))\right|\Bigr)
\end{align*}
in $B_{S_1}(0)$, which also uses \eqref{linfty psir}, \eqref{det d 1} and \eqref{d mod S1}, and the bound \eqref{mod}. This yields that \begin{align*}
& \int_{B_{S_2}(0)} \eta(\mathcal{X}_t^1(x),\mathcal{X}_t^2(x))\left| d^1(t,\mathcal{X}_t^1(x))-d^2(t,\mathcal{X}_t^2(x))\right|\left|(\de_t d^1)(t,\mathcal{X}_t^2(x))-(\de_t d^2)(t,\mathcal{X}_t^2(x))\right||x|^{\beta-6}\dd x\\
&\lesssim \int_{B_{S_2}(0)} \eta(\mathcal{X}_t^1(x),\mathcal{X}_t^2(x))\left| d^1(t,\mathcal{X}_t^1(x))-d^2(t,\mathcal{X}_t^2(x))\right||(\de_t\Psi_R^1)(t,\mathcal{X}_t^2(x))-(\de_t\Psi_R^2)(t,\mathcal{X}_t^2(x))||x|^{\beta-5}\\
&\quad + \eta(\mathcal{X}_t^1(x),\mathcal{X}_t^2(x))\left| d^1(t,\mathcal{X}_t^1(x))-d^2(t,\mathcal{X}_t^2(x))\right||x|^{\beta-4}\dd x.
\end{align*}
The first term can be estimated by $E_t\log(2+\frac{1}{E_t})$ in the same manner as the term $E_t^1$ after using Cauchy-Schwarz, the second term is trivially $\lesssim E_t$. 

Putting all these estimates back together again, we obtain the desired estimate \begin{align*}
\left|\dfrac E_t^3\right|\lesssim E_t\log\left(2+\frac{1}{E_t}\right).
\end{align*}

\subsection{The term $E_t^4$}

We estimate \begin{align*}
\mel\left|\dfrac E_t^4\right|= \left|\dfrac \int_{B_{S_2}(0)} \eta(\mathcal{X}_t^1(x),\mathcal{X}_t^2(x))\mathfrak{d}_t(x)^\beta |\mathcal{X}_t^1(x)_\theta-\mathcal{X}_t^2(x)_\theta|^2\dd x\right|\\
& \lesssim  \int_{B_{S_2}(0)} \left|\nabla \eta(\mathcal{X}_t^1(x),\mathcal{X}_t^2(x))\cdot \binom{\bar u_t^1(\mathcal{X}_t^1(x))}{\bar u_t^2(\mathcal{X}_t^2(x))}\right|
\mathfrak{d}_t(x)^\beta |\mathcal{X}_t^1(x)_\theta-\mathcal{X}_t^2(x)_\theta|^2\\
&\: + \eta(\mathcal{X}_t^1(x),\mathcal{X}_t^2(x))\frac{\max(|(\de_t d^1)(t,\mathcal{X}_t^1(x))|,|(\de_t d^2)(t,\mathcal{X}_t^2(x))|)}{\mathfrak{d}_t(x)}\mathfrak{d}_t(x)^\beta |\mathcal{X}_t^1(x)_\theta-\mathcal{X}_t^2(x)_\theta|^2\\
&\: +\eta(\mathcal{X}_t^1(x),\mathcal{X}_t^2(x))\mathfrak{d}_t(x)^{\beta} |\mathcal{X}_t^1(x)_\theta-\mathcal{X}_t^2(x)_\theta|\left|\bar u_t^1(\mathcal{X}_t^1(x))\cdot \frac{\mathcal{X}_t^1(x)^\perp}{|\mathcal{X}_t^1(x)|^2}-\bar u_t^2(\mathcal{X}_t^2(x))\cdot \frac{\mathcal{X}_t^2(x)^\perp}{|\mathcal{X}_t^2(x)|^2}\right|\dd x,
\end{align*}
where in the last term, we have used the identity $\de_t f_\theta=\de_t f\cdot \frac{f^\perp}{|f|^2}$.

The first two terms are again directly estimated against $E_t$ in the same manner as above.

For the third term, we use 
the estimates \eqref{ll2}, \eqref{mod1} and \eqref{mfrakd} to see that \begin{align}
\mel\int_{B_{S_2}(0)} \eta(\mathcal{X}_t^1(x),\mathcal{X}_t^2(x))\mathfrak{d}_t(x)^{\beta} |\mathcal{X}_t^1(x)_\theta-\mathcal{X}_t^2(x)_\theta|\left|\bar u_t^1(\mathcal{X}_t^1(x))\cdot \frac{\mathcal{X}_t^1(x)^\perp}{|\mathcal{X}_t^1(x)|^2}-\bar u_t^2(\mathcal{X}_t^2(x))\cdot \frac{\mathcal{X}_t^2(x)^\perp}{|\mathcal{X}_t^2(x)|^2}\right|\dd x\nonumber\\
&\lesssim \int_{B_\frac{1}{2}(0)} |x|^{\beta/2-1} \left(k_t(x)\log\left(2+\frac{1}{k_t(x)}\right)+|\bar u_t^1(\mathcal{X}_t^1(x))-\bar u_t^2(\mathcal{X}_t^1(x))|\right)D(\mathcal{X}_t^1(x),\mathcal{X}_t^2(x))\dd x.\label{et4 0}
\end{align}
For the first summand, we use Cauchy-Schwarz to see that \begin{align*}
\mel\int_{B_\frac{1}{2}(0)} |x|^{\beta/2-1} k_t(x)\log\left(2+\frac{1}{k_t(x)}\right)D(\mathcal{X}_t^1(x),\mathcal{X}_t^2(x))\dd x\\
&\lesssim \sqrt{E_t} \left(\int_{B_\frac{1}{2}(0)} |x|^{\beta-2}k_t(x)^2\log\left(2+\frac{1}{k_t(x)}\right)^2\dx\right)^\frac{1}{2}.
\end{align*}
We can then estimate, using Jensen's inequality \begin{align*}
\mel\int_{B_\frac{1}{2}(0)} |x|^{\beta-2}k_t(x)^2\log\left(2+\frac{1}{k_t(x)}\right)^2\dx\\
&\lesssim \int_{B_\frac{1}{2}(0)} |x|^{\beta-2}k_t(x)^2\dx\log\left(2+\frac{1}{(A'\int_{B_\frac{1}{2}(0)}|x|^{\beta-2}k_t(x)^2\dx)^\frac{1}{2}} \right)^2\\
&\lesssim \int_{B_\frac{1}{2}(0)} D(\mathcal{X}_t^1(x),\mathcal{X}_t^2(x))^2\dd x\left(\log\left(2+\frac{1}{\int_{B_\frac{1}{2}(0)} D(\mathcal{X}_t^1(x),\mathcal{X}_t^2(x))^2\dd x}\right) \right)^2\\
&\lesssim E_t\log\left(2+\frac{1}{E_t}\right)^2
\end{align*}
where $A'\approx 1$ is a renormalisation factor  that ensures that $\mathds{1}_{B_\frac{1}{2}(0)}|x|^{\beta-2}$ is a probability measure and the penultimate step used  \eqref{D m}. \allowdisplaybreaks
Hence we have
\begin{align}
    \int_{B_\frac{1}{2}(0)} |x|^{\beta/2-1} k_t(x)\log\left(2+\frac{1}{k_t(x)}\right)D(\mathcal{X}_t^1(x),\mathcal{X}_t^2(x))\dd x\lesssim E_t\log\left(2+\frac{1}{E_t}\right).\label{et4 1}
\end{align}

\noindent For the second summand in \eqref{et4 0}, we split $\bar u_t^1-\bar u_t^2=v_t^1(\cdot)-v_t^1(0)-v_t^2(\cdot)+v_t^2(0)$ and estimate the corresponding summands separately. For the first one, we have by the incompressibility and \eqref{mod1} 
 \begin{align}
 \begin{aligned}
\mel\int_{B_\frac{1}{2}(0)} |x|^{\beta/2-1} \left|v_t^1(\mathcal{X}_t^1(x))-v_t^2(\mathcal{X}_t^1(x))\right|D(\mathcal{X}_t^1(x),\mathcal{X}_t^2(x))\dd x\\
&\lesssim \left(\int_{B_1(0)} |x|^{\beta-2}\left|v_t^1(x)-v_t^2(x)\right|^2\dx\right)^\frac{1}{2}\sqrt{E_t}\lesssim E_t,\label{et4 2}
\end{aligned}\end{align}
where the last step uses Lemma \ref{weighted l2}.

The other summand on the other hand can be estimated by Cauchy-Schwarz as  \begin{align}
\mel \int_{B_\frac{1}{2}(0)} |x|^{\beta/2-1} |v_t^1(0)-v_t^2(0)|D(\mathcal{X}_t^1(x),\mathcal{X}_t^2(x))\dd x\lesssim |v_t^1(0)-v_t^2(0)|\sqrt{E_t} \lesssim E_t\label{et4 3}
\end{align}
since $ |x|^{\beta-2}$ is integrable and where the last estimate is again  Lemma \ref{weighted l2}.

Combining \eqref{et4 0}, \eqref{et4 1}, \eqref{et4 2} and \eqref{et4 3} again, we see that \begin{align*}
    \left|\dfrac E_t^4\right|\lesssim E_t\log\left(2+\frac{1}{E_t}\right).
\end{align*}

\subsection{$\bold{E_t^5}$}
We compute \begin{align}\label{et5 0}\begin{aligned}
 \frac{1}{2}\dfrac E_t^5\leq \int_{B_{S_2}(0)}&\left|\nabla\eta(\mathcal{\Chi}_t^1(x),\mathcal{\Chi}_t^2(x))\cdot \binom{\bar u_t^1(\mathcal{\Chi}^1(x))}{\bar u_t^2(\mathcal{\Chi}^2(x))}\right||\mathcal{\Chi}_t^1(x)-\mathcal{\Chi}_t^2(x)|^2\\
&+(1-\eta)(\mathcal{\Chi}_t^1(x),\mathcal{\Chi}_t^2(x))\left|\scalar{\bar u_t^1(\mathcal{\Chi}_t^1(x))-\bar u_t^2(\mathcal{\Chi}_t^2(x))}{\mathcal{\Chi}_t^1(x)-\mathcal{\Chi}_t^2(x)}\right|\dd x.
\end{aligned}\end{align}
The first term is directly estimated through $\int_{B_{S_2}(0)} |\mathcal{\Chi}_t^1(x)-\mathcal{\Chi}_t^2(x)|^2\dd x\lesssim E_t$ as $\bar u_t^1$ and $\bar u_t^2$ are uniformly bounded away from $0$ by \eqref{uni bd u}.

The second summand in \eqref{et5 0} on the other hand proceeds analogously to the classical uniqueness proof. We can estimate \begin{align*}
\mel\int_{B_{S_2}(0)}(1-\eta)(\mathcal{\Chi}_t^1(x),\mathcal{\Chi}^2(x))\left|\scalar{\bar u_t^1(\mathcal{\Chi}_t^1(x))-\bar u_t^2(\mathcal{\Chi}_t^2(x))}{\mathcal{\Chi}_t^1(x)-\mathcal{\Chi}_t^2(x)}\right|\dd x\\
&\leq\int_{B_{S_2}(0)}(1-\eta)(\mathcal{\Chi}_t^1(x),\mathcal{\Chi}_t^2(x))\left|\bar u_t^1(\mathcal{\Chi}_t^1(x))-\bar u_t^1(\mathcal{\Chi}_t^2(x))\right|\left|\mathcal{\Chi}_t^1(x)-\mathcal{\Chi}_t^2(x)\right|\dd x\\
&\quad+\int_{B_{S_2}(0)}(1-\eta)(\mathcal{\Chi}_t^1(x),\mathcal{\Chi}_t^2(x))\left|\bar u_t^1(\mathcal{\Chi}_t^2(x))-\bar u_t^2(\mathcal{\Chi}_t^2(x))\right|\left|\mathcal{\Chi}_t^1(x)-\mathcal{\Chi}_t^2(x)\right|\dd x\\
&=:I+II.
\end{align*}
%

\noindent The term $I$ is estimated through $\bar u_t^1$ being log-Lipschitz on $\supp 1-\eta$ by Lemma \ref{ll lemma}, yielding, together with Cauchy-Schwarz and Jensen's inequality that \begin{align*}
\mel I\lesssim \left(\int_{B_{S_2}(0)\backslash \Omega} k_t(x)^2\log\left(2+\frac{1}{k_t(x)}\right)^2\dd x\right)^\frac{1}{2}\sqrt{E_t}\\
&\lesssim \left(\int_{B_{S_2}(0)\backslash \Omega} k_t(x)^2 \dd x\right)^\frac{1}{2}\log\left(2+\frac{1}{\left(\fint_{B_{S_2}(0)\backslash \Omega} k_t(x)^2 \dd x\right)^\frac{1}{2}}\right)\sqrt{E_t}\\
&\lesssim \left(\int_{B_{S_2}(0)\backslash \Omega} k_t(x)^2 \dd x\right)^\frac{1}{2}\log\left(2+\frac{1}{\int_{B_{S_2}(0)\backslash \Omega} k_t(x)^2 \dd x}\right)\sqrt{E_t}\\
&\lesssim E_t\log\left(2+\frac{1}{E_t}\right),
\end{align*}
where the last step uses \eqref{D m}.

The term $II$ on the other hand is estimated by Cauchy-Schwarz as \begin{align*}
\mel II\lesssim \sqrt{E_t}\left(\int_{B_{S_2}(0)}|\bar u_t^1(\mathcal{X}_t^2(x))-\bar u_t^2(\mathcal{X}_t^2(x))|^2\dd x\right)^\frac{1}{2}\lesssim \sqrt{E_t}\left(\int_{B_{2S_2}(0)}|\bar u_t^1(x)-\bar u_t^2(x)|^2\dd x\right)^\frac{1}{2},
\end{align*}
where the last step uses the incompressibility and \eqref{mod1}.

We further have \begin{align}
\int_{B_{2S_2}(0)}|\bar u_t^1(x)-\bar u_t^2(x)|^2\dd x\leq 2\int_{B_{2S_2}(0)}|v_t^1(x)-v_t^2(x)|^2+|v_t^1(0)-v_t^2(0)|^2\dd x.\label{velo diff}
\end{align}
We now utilize Lemma \ref{l2 lemma} and the fact that $\varpi_0$ is supported in $B_{S_2}(0)$ to estimate  \begin{align*}
    \int_{B_{S_2}(0)}|v_t^1(x)-v_t^2(x)|^2\dd x\lesssim \norm{\mathcal{X}_t^1-\mathcal{X}_t^2}_{L^2(B_{S_2}(0))}^2\lesssim E_t,
\end{align*}
where the last step uses \eqref{D m} yet again.
For the second term in \eqref{velo diff}, we can directly use \eqref{vt 0} to estimate it with $E_t$, which yields that \begin{align*}
II\lesssim E_t.
\end{align*}
This shows the desired differential estimate $|\dfrac E_t^5|\lesssim E_t\log(2+\frac{1}{E_t})$
and finishes the proof of \eqref{dest}. \vphantom{a}\hfill\qedsymbol

\section{Proof of Lemmas \ref{deri dpsi}-\ref{weighted l2}}\label{S4}
\subsection{Proof of Lemma \ref{deri dpsi}}
We first show \eqref{deri dpsi1}.

We use the definition \eqref{def psir} and integrate by parts to see that \begin{align}
\mel 2\pi\left|\de_t\Psi_R^1(t,x)-\de_t\Psi_R^2(t,x)\right|=\left|\dfrac \int_{\R^2}\left(\log|x-z|-\log|z|+\frac{x\cdot z}{|z|^2}\right) (\varpi_t^1(z)-\varpi_t^2(z))\dd z\right|\nonumber\\
&=\biggl|\int_{\R^2}\left(\log|x-z|-\log|z|+\frac{x\cdot z}{|z|^2}\right) \Bigl(\bar u_t^1(z)\cdot\nabla_z(\varpi_t^1(z)-\varpi_t^2(z))\nonumber\\
&\qquad+(\bar u_t^1(z)-\bar u_t^2(z))\cdot\nabla \varpi_t^2(z)\dd z\biggr|\nonumber\\
&\leq \left|\int_{\R^2}\nabla_z(\log|x-z|-\log|z|+\frac{x\cdot z}{|z|^2})\cdot \bar u_t^1(z)(\varpi_t^1(z)-\varpi_t^2(z))\dd z\right|\nonumber\\
&\quad+\left|\int_{\R^2}\nabla_z(\log|x-z|-\log|z|+\frac{x\cdot z}{|z|^2})\cdot(\bar u_t^1(z)-\bar u_t^2(z))(\varpi_t^2(z)-\tilde{\eta}(z)\varpi_t^2(0))\dd z\right|\nonumber\\
&\quad+\left|\int_{\R^2}\nabla_z(\log|x-z|-\log|z|+\frac{x\cdot z}{|z|^2})\cdot(\bar u_t^1(z)-\bar u_t^2(z))\tilde{\eta}(z)\varpi_t^2(0)\dd z\right|\nonumber\\
&=:P(x)+Q(x)+R(x) \label{def pqr}
\end{align}
where $P, Q$ and $R$ depend on time through the $\bar u$ and $\varpi$ and $\tilde{\eta}$ is a smooth, radially symmetric and nonnegative function supported in $B_2(0)$ which equals $1$ on $B_1(0)$.

Let us treat the term $P$ first.
In it we split the integral into the contributions of $\varpi^1$ and $\varpi^2$ and substitute in the corresponding flows to see that \begin{align}\begin{aligned}\label{def P}
&P(x)\leq \Biggl|\int_{\R^2}\biggl(\left(\frac{(x-\mathcal{X}_t^1(z))}{|x-\mathcal{X}_t^1(z)|^2}-\frac{\mathcal{X}_t^1(z)}{|\mathcal{X}_t^1(z)|^2}+\frac{x|\mathcal{X}_t^1(z)|^2-(x\cdot \mathcal{X}_t^1(z))\mathcal{X}_t^1(z)}{|\mathcal{X}_t^1(z)|^4}\right)\cdot  \bar u_t^1(\mathcal{X}_t^1(z))\\
&-\biggl(\frac{(x-\mathcal{X}_t^2(z))}{|x-\mathcal{X}_t^2(z)|^2}-\frac{\mathcal{X}_t^2(z)}{|\mathcal{X}_t^2(z)|^2}+\frac{x|\mathcal{X}_t^2(z)|^2-(x\cdot \mathcal{X}_t^2(z))\mathcal{X}_t^2(z)}{|\mathcal{X}_t^2(z)|^4}\biggr)\cdot  \bar u_t^1(\mathcal{X}_t^2(z))\biggr)(\varpi_0(z)-\tilde{\eta}(z)\varpi_0(0))\dd z\Biggr|\\
&\qquad+\left|\int_{\R^2}\nabla_z(\log|x-z|-\log|z|+\frac{x\cdot z}{|z|^2})\bar u_t^1(z)\left(\tilde{\eta}\circ (\mathcal{X}_t^1)^{-1}(z)+\tilde{\eta}\circ (\mathcal{X}_t^2)^{-1}(z)\right)\varpi_0(0)\dd z \right|\\
&=:P'(x)+P''(x).
\end{aligned}\end{align}
We introduce the standard dyadic annuli \begin{align}
A_n:=B_{2^{-n}}(0)\backslash B_{2^{-(n+1)}}(0). \label{def annuli}
\end{align}
We also introduce the fattened annuli and the inner and outer regions \begin{align}I_n:=\{0\}\cup\!\!\bigcup_{m\geq n+2} A_m,\quad\bar{A}_n:=A_{n+1}\cup A_n\cup A_{n-1},\quad \tilde{A}_n:=\bigcup_{i=-2}^2 A_{n+i},\quad  O_n:=\bigcup_{-\lceil \log_2 S_2\rceil\leq m\leq n-2} A_m.\label{def regions}
\end{align}
%

\noindent We now realize that the point vortex velocity is orthogonal to the $x$-independent terms and obtain that the integrand in $P'$ can be rewritten as \allowdisplaybreaks \begin{align*}
&\biggl(\left(\frac{(x-\mathcal{X}_t^1(z))}{|x-\mathcal{X}_t^1(z)|^2}-\frac{\mathcal{X}_t^1(z)}{|\mathcal{X}_t^1(z)|^2}+\frac{x|\mathcal{X}_t^1(z)|^2-(x\cdot \mathcal{X}_t^1(z))\mathcal{X}_t^1(z)}{|\mathcal{X}_t^1(z)|^4}\right)\cdot  \bar u_t^1(\mathcal{X}_t^1(z))\\
&\quad-\left(\frac{(x-\mathcal{X}_t^2(z))}{|x-\mathcal{X}_t^2(z)|^2}-\frac{\mathcal{X}_t^2(z)}{|\mathcal{X}_t^2(z)|^2}+\frac{x|\mathcal{X}_t^2(z)|^2-(x\cdot \mathcal{X}_t^2(z))\mathcal{X}_t^2(z)}{|\mathcal{X}_t^2(z)|^4}\right)\cdot  \bar u_t^1(\mathcal{X}_t^2(z)))\biggr)(\varpi_0(z)-\tilde{\eta}(z)\varpi_0(0))\\
&=\biggl(\frac{(x-\mathcal{X}_t^1(z))}{|x-\mathcal{X}_t^1(z)|^2}-\frac{(x-\mathcal{X}_t^2(z))}{|x-\mathcal{X}_t^2(z)|^2}\biggr)\cdot \bar u_t^1(\mathcal{X}_t^1(z))(\varpi_0(z)-\tilde{\eta}(z)\varpi_0(0))\\
&\quad+\frac{(x-\mathcal{X}_t^2(z))}{|x-\mathcal{X}_t^2(z)|^2}\cdot\bigl(\bar u_t^1(\mathcal{X}_t^1(z))-\bar u_t^1(\mathcal{X}_t^2(z))\bigr)(\varpi_0(z)-\tilde{\eta}(z)\varpi_0(0))\\
&\quad +\left(\frac{x|\mathcal{X}_t^1(z)|^2-(x\cdot \mathcal{X}_t^1(z))\mathcal{X}_t^1(z)}{|\mathcal{X}_t^1(z)|^4}-\frac{x|\mathcal{X}_t^2(z)|^2-(x\cdot \mathcal{X}_t^2(z))\mathcal{X}_t^2(z)}{|\mathcal{X}_t^2(z)|^4}\right)\cdot  \bar u_t^1(\mathcal{X}_t^1(z)) (\varpi_0(z)-\tilde{\eta}\varpi_0(0))\\
&\quad +\frac{x|\mathcal{X}_t^2(z)|^2-(x\cdot \mathcal{X}_t^2(z))\mathcal{X}_t^2(z)}{|\mathcal{X}_t^2(z)|^4}\cdot\bigl(\bar u_t^1(\mathcal{X}_t^1(z))-\bar u_t^1(\mathcal{X}_t^2(z))\bigr)(\varpi_0(z)-\tilde{\eta}(z)\varpi_0(0))\\
&\quad -\left(\frac{\mathcal{X}_t^1(z)}{|\mathcal{X}_t^1(z)|^2}-\frac{\mathcal{X}_t^2(z)}{|\mathcal{X}_t^2(z)|^2}\right)\cdot (v_t^1(\mathcal{X}_t^1(z))-v_t^1(0))(\varpi_0(z)-\tilde{\eta}(z)\varpi_0(0))\\
&\quad -\frac{\mathcal{X}_t^2(z)}{|\mathcal{X}_t^2(z)|^2}\cdot \left(v_t^1(\mathcal{X}_t^1(z))-v_t^1(\mathcal{X}_t^2(z))\right)(\varpi_0(z)-\tilde{\eta}(z)\varpi_0(0))\\
&=:\sum_{i=1}^6 p_i(x,z).
\end{align*}
%

We will estimate the weighted $L^2$ norm of $P'$ as follows: \begin{align}\begin{aligned}\label{splitting p'}
\mel\norm{|x|^{\beta/2-2}P'}_{L^2(B_\frac{1}{2}(0))}\\
&\lesssim \sum_{i=1}^6 \norm{\sum_{m\geq 1}\mathds{1}_{A_m}(x)|x|^{\beta/2-2}\int_{\bar{A}_m}p_i(x,z)\dd z}_{L^2(\R^2)}+\norm{\sum_{m\geq 1}\mathds{1}_{A_m}(x)|x|^{\beta/2-2}\int_{I_m}p_i(x,z)\dd z}_{L^2(\R^2)}\\
&\quad+\norm{\sum_{m\geq 1}\mathds{1}_{A_m}(x)|x|^{\beta/2-2}\int_{O_m}(p_1+p_3+p_5)(x,z)\dd z}_{L^2(\R^2)}\\
&\quad+\norm{\sum_{m\geq 1}\mathds{1}_{A_m}(x)|x|^{\beta/2-2}\int_{O_m}(p_2+p_4+p_6)(x,z)\dd z}_{L^2(\R^2)},
\end{aligned}\end{align}
which uses that both $\varpi_0$ and $\tilde{\eta}$ are $0$ outside of $B_{S_2}(0)\subset I_m\cup \bar{A}_m\cup O_m$.
Before estimating these terms, let us point out that a pointwise estimate of the type \begin{align}
    \left|\int_{\text{$\bar{A}_n$ or $I_n$}}p_i(x,z)\dd z \right|\lesssim |x|\sqrt{E_t}\log\left(2+\frac{1}{E_t}\right)\label{app pw est}
\end{align}
for $x\in A_n$ implies that the weighted $L^2$-norm of the corresponding summand has the desired bound (i.e. $\norm{\sum_{m\geq 1}\mathds{1}_{A_m}(x)|x|^{\beta/2-2}\int_{\text{$\bar{A}_m$ or $I_m$}}p_i(x,z)\dd z}_{L^2(\R^2)}\lesssim \sqrt{E_t}\log(2+1/E_t)$), as one sees directly from summing up. The same also holds for the integrals over $O_m$.

We now estimate these $14$ terms separately. 
We first note that on $I_n$ and $\bar{A}_n$ it holds $\tilde{\eta}=1$.
We rewrite, using the substitution rule \begin{align*}
\left|\int_{\bar{A}_n}p_1(x,z)\dd z\right|=\left|\int_{\R^2}\frac{(x-z)}{|x-z|^2}\cdot \left(\widetilde{\varpi}(n)\circ (\mathcal{X}_t^1)^{-1}(z)-\widetilde{\varpi}(n)\circ (\mathcal{X}_t^2)^{-1}(z)\right)\dd z\right|
\end{align*}
where $\widetilde{\varpi}(n)(y):=\mathds{1}_{\bar{A}_n}\bar u_t^1(\mathcal{X}_t^1(y))(\varpi_0(y)-\varpi_0(0))$. 

We may now recognize that this convolution is the Biot-Savart law and apply the $L^2$-estimate from Lemma \ref{l2 lemma} to this to see that \begin{align*}
\norm{\int_{\R^2}\frac{(x-z)}{|x-z|^2}\cdot \left(\tilde{\varpi}(n)\circ (\mathcal{X}_t^1)^{-1}-\tilde{\varpi}(n)\circ (\mathcal{X}_t^2)^{-1}\right)\dd z}_{L_x^2(A_n)}\lesssim \norm{\mathcal{X}_t^1-\mathcal{X}_t^2}_{L^2(\tilde{A}_n)}\norm{\tilde{\varpi}(n)}_{L^\infty},
\end{align*}
where we have also used that by \eqref{mod1}, the preimage of $\bar{A}_n$ under both $\mathcal{X}_t^1$ and $\mathcal{X}_t^2$ is contained in $\tilde{A}_n$.

We have \begin{align*}
 \norm{\widetilde{\varpi}(n)}_{L^\infty} \lesssim  2^{-(\alpha-1)n}
\end{align*}
because $|\bar u_t^1(\mathcal{X}_t^1(z))|\lesssim |z|^{-1}$ by \eqref{uni bd u} and \eqref{mod1} and because of the estimate \eqref{ass later t}.

In sum, we obtain that \begin{align*}\begin{aligned}
\mel \norm{\sum_{m\geq 1} \left|\mathds{1}_{x\in A_m}|x|^{-2+\frac{\beta}{2}}\int_{\bar{A}_m}p_1(x,z)\dd z\right|}_{L^2}^2\lesssim \sum_{m\geq 1} 2^{(4-\beta)m}\norm{\mathcal{X}_t^1-\mathcal{X}_t^2}_{L^2(\tilde{A}_m)}^2\norm{\tilde{\varpi}(m)}_{L^\infty}^2\\
&\lesssim \sum_{m\geq -1} 2^{-2(\alpha-2)m}\norm{D(\mathcal{X}_t^1,\mathcal{X}_t^2)}_{L^2(A_n)}^2  
\lesssim E_t,
\end{aligned}\end{align*}
where the penultimate step used \eqref{D m} and the last one used that $\alpha>2$.

On the region $I_n$, we can simply estimate the integral of $p_1$ pointwise. For $x\in A_n$, we have \begin{align*}
\left|\int_{I_n}p_1(x,z)\dd z\right|\lesssim \left|\int_{I_n} \frac{|\mathcal{X}_t^1(z)-\mathcal{X}_t^2(z)|}{|x|^2}|\bar u_t^1(\mathcal{X}_t^1(z))||\varpi_0(z)-\varpi_0(0)|\dd z\right|
\end{align*}
because $\dist(x,\mathcal{X}_t^1(I_n))\approx \dist(x,\mathcal{X}_t^2(I_n))\approx x$ by definition and \eqref{mod1}. Since we also have \begin{align*}
|\bar u_t^1(\mathcal{X}_t^1(z))||\varpi_0(z)-\varpi_0(0)|\lesssim |z|^{\alpha-1}
\end{align*}
by \eqref{ass later t}, \eqref{uni bd u} and \eqref{mod1}, we see that \begin{align*}
\mel\int_{I_n} \frac{|\mathcal{X}_t^1(z)-\mathcal{X}_t^2(z)|}{|x|^2}|\bar u_t^1(\mathcal{X}_t^1(z))||\varpi_0(z)-\varpi_0(0)|\dd z\lesssim \int_{I_n} \frac{|\mathcal{X}_t^1(z)-\mathcal{X}_t^2(z)|}{|x|^2}|z|^{\alpha-1}\dd z\\
&\lesssim \sqrt{E_t}|x|^{\frac{2\alpha-\beta-2}{2}},
\end{align*}
where the last step uses Cauchy-Schwarz, that $I_n$ is a ball of radius $|x|$, and \ref{D m}. Upon noting that $\frac{2\alpha-\beta-2}{2}>1$, we note that the pointwise estimate \eqref{app pw est} holds for this term.


\noindent We proceed to the next term, by \eqref{mod1}, \eqref{ass later t} and \eqref{ll1}, we have, for $x\in A_n$,  \begin{align*}
\mel\left|\int_{\bar{A}_n}p_2(x,z)\dd z\right|\lesssim \int_{\bar{A}_n} \frac{1}{|x-\mathcal{X}_t^2(z)|}k_t(z)\log\left(2+\frac{1}{k_t(z)}\right)|z|^{\alpha-2}\dd z\\
&\approx |x|^{\alpha-1-\beta/2}\int_{(\mathcal{X}_t^2)^{-1}\bar{A}_n} \frac{1}{|x-z|}\left(k_t\log\left(2+\frac{1}{k_t}\right)\right)\circ (\mathcal{X}_t^2)^{-1}(z)|z|^{\beta/2-1}\dd z,
\end{align*}
where the last step uses the substitution rule and that $|z|\approx |x|$ on this region.

To estimate this integral, we use Young's convolution inequality to see that \begin{align*}
\mel\norm{\int_{(\mathcal{X}_t^2)^{-1}\bar{A}_n} \frac{|z|^{\beta/2-1}}{|x-z|}\left(k_t\log\left(2+\frac{1}{k_t}\right)\right)\circ (\mathcal{X}_t^2)^{-1}(z)\dd z}_{L_x^2(A_n)}\\
&\lesssim \norm{\mathds{1}_{B_{2^{-(n-5)}(0)}}\frac{1}{|\cdot|}}_{L^1}\norm{|\cdot|^{\beta/2-1}k_t\log\left(2+\frac{1}{k_t}\right)\circ (\mathcal{X}_t^2)^{-1}}_{L^2(\bar{A}_n)}\\
&\lesssim2^{-n}\norm{|\cdot|^{\beta/2-1}k_t\log\left(2+\frac{1}{k_t}\right)}_{L^2(\tilde{A}_n)},
\end{align*}
where we have used that $(\mathcal{X}_t^2)^{-1}\bar{A}_n\subset \tilde{A}_n$ by \eqref{mod1} and that for $z\in \tilde{A}_n$ and $x\in \bar{A}_n$ it holds that $|x-z|\leq 2^{-(n-5)}$.

To further estimate this $L^2$-norm, we apply Jensen's inequality with the concave function $y\log(2+y^{-\frac{1}{2}})^2$ and the probability measure $|\tilde{A}_n|^{-1}\mathds{1}_{\tilde{A}_n}$ to see that \begin{align}\begin{aligned}
\mel\norm{|\cdot|^{\beta/2-1}k_t\log\left(2+\frac{1}{k_t}\right)}_{L^2(\tilde{A}_n)}\lesssim \left(\int_{\tilde{A_n}}k_t(z)^2|z|^{\beta-2}\dd z\right)^\frac{1}{2}\log\left(2+\frac{1}{\left(\fint_{\tilde{A}_n}k_t(z)^2|z|^{\beta-2}\dd z\right)^\frac{1}{2}}\right)\\
&\lesssim \sqrt{E_t}\log\left(2+\frac{2^{n}}{\sqrt{E_t}}\right)\lesssim \sqrt{E_t}\left(n+\log\left(2+\frac{1}{E_t}\right)\right),\label{jensens}
\end{aligned}\end{align}
where we have used \eqref{D m}.

Combining the previous estimates, we see that \begin{align*}
 \mel   \norm{\sum_{m\geq 1}\mathds{1}_{A_m}(x)|x|^{\beta/2-2}\left|\int_{\bar{A}_m}p_2(x,z)\dd z\right|}_{L_x^2}^2\lesssim \sum_{m\geq 0} 2^{-2(\alpha-\beta/2-2+\beta/2)m}E_t\left(n^2+\log\left(2+\frac{1}{E_t}\right)^2\right)\\
    &\lesssim E_t \log\left(2+\frac{1}{E_t}\right)^2.
\end{align*}

\noindent In the inner region we can estimate pointwise, for $x\in A_n$, using the bounds \eqref{uni bd u}, \eqref{ass later t} and \eqref{mod1} and the fact that by \eqref{mod1} and the definition of $I_n$ we have that $|x-\mathcal{X}_t^1(z)|\approx |x|$ \begin{align*}
   \mel \left|\int_{I_n}p_2(x,z)\dd z\right|\lesssim \int_{I_n} |x|^{-1}|z|^{\alpha-2} k_t(z)\log\left(2+\frac{1}{k_t(z)}\right)\dd z\\
   &\lesssim |x|^{\alpha-2-\frac{\beta}{2}}\int_{I_n} k_t(z)\log\left(2+\frac{1}{k_t(z)}\right)|z|^{\frac{\beta}{2}-1}\dd z\\
   &\lesssim |x|^{\alpha-1-\frac{\beta}{2}}\left(\int_{I_n} k_t(z)^2\log\left(2+\frac{1}{k_t(z)}\right)^2|z|^{\beta-2}\dd z\right)^\frac{1}{2},
\end{align*}
where the last step is Cauchy-Schwarz combined with the fact that $|I_n|\approx |x|^2$. 
Applying Jensen's inequality as in \eqref{jensens} above, we see that \begin{align*}
    \left(\int_{I_n} k_t(z)^2\log\left(2+\frac{1}{k_t(z)}\right)^2|z|^{\beta-2}\dd z\right)^\frac{1}{2}\lesssim E_t\left(|\log|x||+\log\left(2+\frac{1}{E_t}\right)\right),
\end{align*}
giving us the desired bound of the type \eqref{app pw est} over this region since $|x|^{\alpha-1-\frac{\beta}{2}}|\log|x||\lesssim |x|$.

We move on to the term $p_3$ for $x\in A_n$, here one can treat the regions $I_n$ and $\bar{A}_n$ simultaneously.
Estimating the difference with the derivative and using that $|\mathcal{X}_t^1(z)|\approx |\mathcal{X}_t^2(z)|\approx |z|$ by \eqref{mod1}, we see that \begin{align*}
  \mel  \biggl|\frac{x|\mathcal{X}_t^1(z)|^2-(x\cdot \mathcal{X}_t^1(z))\mathcal{X}_t^1(z)}{|\mathcal{X}_t^1(z)|^4}-\frac{x|\mathcal{X}_t^2(z)|^2-(x\cdot \mathcal{X}_t^2(z))\mathcal{X}_t^2(z)}{|\mathcal{X}_t^2(z)|^4}\biggr|\lesssim|x||z|^{-3}|\mathcal{X}_t^1(z)-\mathcal{X}_t^2(z)|\\
    &\lesssim  |x||z|^{-3}k_t(z).
\end{align*}
Together with \eqref{uni bd u}, \eqref{ass later t}, \eqref{mod1} and \eqref{D m}, this implies that \begin{align*}
    |p_3(x,z)|\lesssim |x||z|^{\alpha-4}k_t(z)\lesssim |x||z|^{\alpha-3-\frac{\beta}{2}}D(\mathcal{X}_t^1(z),\mathcal{X}_t^2(z)).
\end{align*}
Since $\alpha-3-\frac{\beta}{2}>\frac{\beta}{2}-1$, we conclude by Cauchy-Schwarz and the fact that $I_n\cup \bar{A}_n$ is a ball of radius $\approx |x|$ that \begin{align*}
\int_{I_n\cup \bar{A}_n}|p_3(x,z)|\dd z\lesssim |x|\sqrt{E_t},
\end{align*}
which is a pointwise bound of the type \eqref{app pw est}.

We treat $p_4$ next, again for $x\in A_n$, for which $I_n$ and $\bar{A}_n$ can be treated together too. Using yet again \eqref{ll1} and \eqref{mod1}, we see that we have \begin{align*}
    |p_4(x,z)|\lesssim |x||z|^{\alpha-4} k_t(z)\log \left(2+\frac{1}{k_t(z)}\right),
\end{align*}
and hence we can estimate \begin{align*}
\mel \left|\int_{I_n\cup \bar{A}_n}p_4(x,z)\dd z\right|\lesssim \int_{I_n\cup \bar{A}_n}|x||z|^{\alpha-4} k_t(z)\log \left(2+\frac{1}{k_t(z)}\right)\dd z\\
& \lesssim|x|^{\alpha-1-\beta/2}\left(\int_{I_n\cup \bar{A}_n} k_t(z)^2\log \left(2+\frac{1}{k_t(z)}\right)^2|z|^{\beta-2}\dd z\right)^\frac{1}{2}
\end{align*}
where the second step consists of applying Cauchy-Schwarz and estimating $|z|\lesssim |x|$.

This can be estimated by $\lesssim |x|^{\alpha-1-\beta/2}\sqrt{E_t}\left(|\log|x||+\log\left(2+\frac{1}{E_t}\right)\right)$ by the same argument as in \eqref{jensens} and in particular has a pointwise bound of the type \eqref{app pw est} as $\alpha-1-\beta/2>1$.

We move on to the $x$-independent terms $p_5$ and $p_6$, for which one can handle the regions $I_n\cup \bar{A}_n$ together as well for $x\in A_n$. Using that by \eqref{mod1}, we have \begin{align*}
\left|\frac{\mathcal{X}_t^1(z)}{|\mathcal{X}_t^1(z)|^2}-\frac{\mathcal{X}_t^2(z)}{|\mathcal{X}_t^2(z)|^2}\right|\lesssim|z|^{-2}|\mathcal{X}_t^1(z)-\mathcal{X}_t^2(z)|\lesssim  |z|^{-2}k_t(z),
\end{align*}
we can estimate, using \eqref{ass later t} and the bound \eqref{est v} on $v_t^1$ \begin{align*}
   \int_{I_n\cup \bar{A}_n} |p_5(x,z)|\dd z\lesssim \int_{I_n\cup \bar{A}_n}k_t(z)|z|^{\alpha-1}\dd z\lesssim |x|^{\alpha-\frac{\beta}{2}}\int_{I_n\cup \bar{A}_n}k_t(z)|z|^{\frac{\beta}{2}-1}\dd z\lesssim |x|^{\alpha-\frac{\beta}{2}}\sqrt{E_t}
\end{align*}
where the last step is Cauchy-Schwarz, combined with \eqref{D m}. This has the desired pointwise bound \eqref{app pw est}.

For the last summand $p_6$, we use that $v_t$ is log-Lipschitz by Lemma \ref{ll lemma app} and use the bounds \eqref{ass later t} and \eqref{mod1} to estimate \begin{align*}
    \int_{I_n\cup\bar{A}_n}|p_6(x,z)|\dd y\lesssim \int_{I_n\cup\bar{A}_n} |z|^{\alpha-1}k_t(z)\log\left(2+\frac{1}{k_t(z)}\right)\dd z.
\end{align*}
This term can be estimated similarly as the one in \eqref{jensens}, yielding a bound of the type \eqref{app pw est} (in fact, it is even an order smaller).

\subsubsection{The integrals over $O_n$ in $P'$}
We now proceed to the outer $O_n$ in the estimate of $P$, here we do split things differently in order to recover the second-order zero of $\Psi_R^1$. We regroup \begin{align}\begin{aligned}\label{rewr p1p3}
\mel p_1(x,z)+p_3(x,z)+p_5(x,z)=\biggl(\frac{(x-\mathcal{X}_t^1(z))}{|x-\mathcal{X}_t^1(z)|^2}-\frac{\mathcal{X}_t^1(z)}{|\mathcal{X}_t^1(z)|^2}+\frac{x|\mathcal{X}_t^1(z)|^2-(x\cdot \mathcal{X}_t^1(z))\mathcal{X}_t^1(z)}{|\mathcal{X}_t^1(z)|^4}\\
&\quad\!-\frac{(x-\mathcal{X}_t^2(z))}{|x-\mathcal{X}_t^2(z)|^2}+\frac{\mathcal{X}_t^2(z)}{|\mathcal{X}_t^2(z)|^2}-\frac{x|\mathcal{X}_t^2(z)|^2-(x\cdot \mathcal{X}_t^2(z))\mathcal{X}_t^2(z)}{|\mathcal{X}_t^2(z)|^4}\biggr)  \cdot \bar u_t^1(\mathcal{X}_t^1(z))\left(\varpi_0(z)-\tilde{\eta}\varpi_0(0)\right)
\end{aligned}\end{align}
and \begin{align}\label{rewr p2p4}
\begin{aligned}p_2(x,z)+p_4(x,z)+p_6(x,z)=& \left(\frac{(x-\mathcal{X}_t^2(z))}{|x-\mathcal{X}_t^2(z)|^2}-\frac{\mathcal{X}_t^2(z)}{|\mathcal{X}_t^2(z)|^2}+\frac{x|\mathcal{X}_t^2(z)|^2-(x\cdot \mathcal{X}_t^2(z))\mathcal{X}_t^2(z)}{|\mathcal{X}_t^2(z)|^4}\right)\\
&\cdot \left(\bar u_t^1(\mathcal{X}_t^1(z))-\bar u_t^1(\mathcal{X}_t^2(z))\right)\left(\varpi_0(z)-\tilde{\eta}(z)\varpi_0(0)\right).
\end{aligned}\end{align}

%
\noindent In the first group, we claim that for $x\in A_n$ and $z\in O_n$ it holds that \begin{align}\begin{aligned}
\mel \biggl|\frac{(x-\mathcal{X}_t^1(z))}{|x-\mathcal{X}_t^1(z)|^2}-\frac{\mathcal{X}_t^1(z)}{|\mathcal{X}_t^1(z)|^2}+\frac{x|\mathcal{X}_t^1(z)|^2-(x\cdot \mathcal{X}_t^1(z))\mathcal{X}_t^1(z)}{|\mathcal{X}_t^1(z)|^4}\\
&\qquad-\frac{(x-\mathcal{X}_t^2(z))}{|x-\mathcal{X}_t^2(z)|^2}+\frac{\mathcal{X}_t^2(z)}{|\mathcal{X}_t^2(z)|^2}-\frac{x|\mathcal{X}_t^2(z)|^2-(x\cdot \mathcal{X}_t^2(z))\mathcal{X}_t^2(z)}{|\mathcal{X}_t^2(z)|^4}\biggr|\lesssim |x||z|^{\alpha-3}k_t(z).\label{claim p1p3}
\end{aligned}\end{align}
We will need to distinguish the cases $|\mathcal{X}_t^1(z)-\mathcal{X}_t^2(z)|\leq \frac{1}{100}|z|$ and $|\mathcal{X}_t^1(z)-\mathcal{X}_t^2(z)|\geq \frac{1}{100}|z|$. 

In the first case, we may estimate, using the shorthand $g(p,q):=\log(|p-q|^2)$ \begin{align*}
\mel \biggl|\frac{(x-\mathcal{X}_t^1(z))}{|x-\mathcal{X}_t^1(z)|^2}-\frac{\mathcal{X}_t^1(z)}{|\mathcal{X}_t^1(z)|^2}+\frac{x|\mathcal{X}_t^1(z)|^2-(x\cdot \mathcal{X}_t^1(z))\mathcal{X}_t^1(z)}{|\mathcal{X}_t^1(z)|^4}\\
&\qquad-\frac{(x-\mathcal{X}_t^2(z))}{|x-\mathcal{X}_t^2(z)|^2}+\frac{\mathcal{X}_t^2(z)}{|\mathcal{X}_t^2(z)|^2}-\frac{x|\mathcal{X}_t^2(z)|^2-(x\cdot \mathcal{X}_t^2(z))\mathcal{X}_t^2(z)}{|\mathcal{X}_t^2(z)|^4}\biggr|\\
&= \Bigl| \nabla_q g(x,\mathcal{X}_t^1(z))-\nabla_q g(0,\mathcal{X}_t^1(z))-\nabla_{pq}^2(0,\mathcal{X}_t^1(z))\cdot x\\
&\quad-\left(\nabla_q g(x,\mathcal{X}_t^2(z))-\nabla_q g(0,\mathcal{X}_t^2(z))-\nabla_{pq}^2(0,\mathcal{X}_t^2(z))\cdot x\right)\Bigr|\\
&=\biggl|\int_0^1\int_0^1\left( \nabla_{pqq}^3 g(s_1\frac{x}{|x|},(1-s_2)\mathcal{X}_t^1(z)+s_2\mathcal{X}_t^2(z))-\nabla_{pqq}^3g(0,(1-s_2)\mathcal{X}_t^1(z)+s_2\mathcal{X}_t^2(z))\right)\\
&\quad:x\otimes (\mathcal{X}_t^1(z)-\mathcal{X}_t^2(z))\dd s_1\dd s_2\biggr|
\end{align*}
by applying the fundamental theorem of calculus twice. 
We now note that the line between $\mathcal{X}_t^1(z)$ and $\mathcal{X}_t^2(z)$ lies outside of $B_{\frac{3}{4}|z|}(0)$ thanks to \eqref{mod1} and the assumption in the case distinction, in particular it has a distance $\approx |z|$ from $B_{|x|}(0)$ and therefore we can estimate \begin{align*}
\mel\biggl|\int_0^1\int_0^1\left( \nabla_{pqq}^3 g(s_1\frac{x}{|x|},(1-s_2)\mathcal{X}_t^1(z)+s_2\mathcal{X}_t^2(z))-\nabla_{pqq}^3g(0,(1-s_2)\mathcal{X}_t^1(z)+s_2\mathcal{X}_t^2(z))\right)\\
&\quad:x\otimes (\mathcal{X}_t^1(z)-\mathcal{X}_t^2(z))\dd s_1\dd s_2\biggr|\\
&\lesssim |x||z|^{-3}k_t(z). 
\end{align*}
The other case $|\mathcal{X}_t^1(z)-\mathcal{X}_t^2(z)|\geq \frac{1}{100}|z|$ proceeds similarly, except we only use the fundamental theorem of calculus once to estimate \begin{align*}
   & \Bigl| \nabla_q g(x,\mathcal{X}_t^1)-\nabla_q g(0,\mathcal{X}_t^1)-\nabla_{pq}^2(0,\mathcal{X}_t^1)\cdot x-\left(\nabla_q g(x,\mathcal{X}_t^2)-\nabla_q g(0,\mathcal{X}_t^2)-\nabla_{pq}^2(0,\mathcal{X}_t^2)\cdot x\right)\Bigr|\\
   &\lesssim \int_0^1\left| \nabla_{pq}^2 g(s_1\frac{x}{|x|},\mathcal{X}_t^1(z))-\nabla_{pq}^{2}g(0,\mathcal{X}_t^1(z))\right||x|+\left| \nabla_{pq}^2 g(s_1\frac{x}{|x|},\mathcal{X}_t^2(z))-\nabla_{pq}^{2}g(0,\mathcal{X}_t^2(z))\right||x|\dd s\\
   &\lesssim |x||z|^{-2}\lesssim |x||z|^{-3}k_t(z) 
\end{align*} 
where we have again used that $\mathcal{X}_t^1(z),\mathcal{X}_t^2(z)$ are outside of $B_{|x|}(0)$ and used the assumption in the case distinction in the last step.
This shows the claim \eqref{claim p1p3}.

We have that $|\varpi_0(z)-\tilde{\eta}(z)\varpi_0(0)|\lesssim |z|^\alpha$ by \eqref{ass later t} and since $\tilde{\eta}=1$ near 0.
Combining this with the rewriting \eqref{rewr p1p3} and the bounds \eqref{claim p1p3}, \eqref{uni bd u} and \eqref{mod1} shows that \begin{align*}
\left|\int_{O_n} p_1(x,z)+p_3(x,z)+p_5(x,z)\dd z\right|\lesssim |x|\left|\int_{O_n}k_t(z)|z|^{\alpha-4}\dd z \right|\lesssim \sqrt{E_t}\
\end{align*}
where the second step uses \eqref{D m} and that $\alpha-2-\beta>0$.
This has a pointwise bound of the type \eqref{app pw est}.

We continue with the second group $p_2(x,z)+p_4(x,z)+p_6(x,z)$. Here we can estimate, for $x\in A_n$ and $z\in O_n$, similarly as in \eqref{claim p1p3} above \begin{align*}
\left|\frac{(x-\mathcal{X}_t^2(z))}{|x-\mathcal{X}_t^2(z)|^2}-\frac{\mathcal{X}_t^2(z)}{|\mathcal{X}_t^2(z)|^2}+\frac{x|\mathcal{X}_t^2(z)|^2-(x\cdot \mathcal{X}_t^2(z))\mathcal{X}_t^2(z)}{|\mathcal{X}_t^2(z)|^4}\right|\lesssim |x||z|^{-2},
\end{align*}
which, together with the log-Lipschitz property \eqref{ll1} of $\bar u_t^1$, \eqref{ass later t}, \eqref{mod1} and the rewritten form \eqref{rewr p2p4} yields  \begin{align*}
\mel \left|\int_{O_n} p_2(x,z)+p_4(x,z)+p_6(x,z)\dd z\right|\leq |x|\int_{O_n} k_t(z)\log\left(2+\frac{1}{k_t(z)}\right)|z|^{\alpha-4}\dd z\\
&\lesssim |x|^{\alpha-1-\beta}\left(\int_{O_n} k_t(z)^2\log\left(2+\frac{1}{k_t(z)}\right)^2|z|^{\beta-2}\dd z\right)^\frac{1}{2},
\end{align*}
where the second step uses Cauchy-Schwarz and that $|z|\lesssim 1$. Jensen's inequality then shows that
\begin{align*}
    \left|\int_{O_n} p_2(x,z)+p_4(x,z)+p_6(x,z)\dd z\right|\lesssim |x|\sqrt{E_t}\log\left(2+\frac{1}{E_t}\right),
\end{align*}
which is a bound of the desired type \eqref{app pw est}.

Coming back to the splitting \eqref{splitting p'}, all these estimates combined show that \begin{align}\label{fin est p'}
    \norm{|\cdot|^{\beta/2-2}P'}_{L^2(B_\frac{1}{2}(0))}\lesssim \sqrt{E_t}\log\left(2+\frac{1}{E_t}\right).
\end{align}

\subsubsection{The term $P''$}
Let us note that by \eqref{mod1}, the difference $\tilde{\eta}\circ (\mathcal{X}_t^1)^{-1}-\tilde{\eta}\circ (\mathcal{X}_t^2)^{-1}$ in the definition \eqref{def P} of $P''$ can only be nonzero if $|z|\in (\frac{3}{4},3)$. We note that for $|x|\leq \frac{1}{2}$ and $|z|\in (\frac{3}{4},3)$ we have \begin{align*}
    \left|\nabla_z(\log|x-z|-\log|z|+\frac{x\cdot z}{|z|^2})\right|+\left|\nabla_z^2(\log|x-z|-\log|z|+\frac{x\cdot z}{|z|^2})\right|\lesssim |x|^2,
\end{align*}
uniformly in $x$ and $z$ as one sees from the fundamental theorem of calculus. We also have that $u_t^1$ is log-Lipschitz on that region by \eqref{ll1}.

We can hence take a  function $h_t=h_t(x,z)$ which agrees with $\nabla_z(\log|x-z|-\log|z|+\frac{x\cdot z}{|z|^2})\cdot u_t^1(z)$ on $B_{\frac{1}{2}}(0)\times \left(B_3(0)\backslash B_{\frac{3}{4}}(0)\right)$ and which enjoys the bound \begin{align*}|h_t(x,z)|+\frac{|h_t(x,z)-h_t(x,z')|}{|z-z'|\log(2+\frac{1}{|z-z'|})}\lesssim |x|^2
\end{align*}
on all of $B_1(0)\times B_3(0)$, uniformly in $t$. We therefore see that \begin{align*}
\mel P''(x)=\left|\int_{\R^2}h_t(x,z)\left(\tilde{\eta}\circ (\mathcal{X}_t^1)^{-1}(z)-\tilde{\eta}\circ (\mathcal{X}_t^2)^{-1}(z)\right)\varpi_0(0)\dd z \right|\\
&=\left|\int_{\R^2}\left(h_t(x,\mathcal{X}_t^1(z))-h_t(x,\mathcal{X}_t^2(z))\right)\tilde{\eta}(z)\varpi_0(0)\dd z \right|\\
&\lesssim |x|^2\int_{\R^2}\left|\mathcal{X}_t^1(z)-\mathcal{X}_t^2(z)\right|\log\left(2+\frac{1}{\left|\mathcal{X}_t^1(z)-\mathcal{X}_t^2(z)\right|}\right)\tilde{\eta}(z)\dd z.
\end{align*}
Using Jensen's inequality with the concave function $y\log(2+y^{-1})$ and Cauchy-Schwarz, we see that \begin{align*}
    P''(x)\lesssim |x|^2\sqrt{E_t}\log\left(2+\frac{1}{E_t}\right),
\end{align*}
yielding the desired bound \begin{align}
    \norm{|\cdot|^{\beta/2-2}P''}_{L^2(B_{\frac{1}{2}}(0))}\lesssim \sqrt{E_t}\log\left(2+\frac{1}{E_t}\right).\label{fin est p''}
\end{align}

\subsubsection{The term $Q$}
We move on to the analysis of the term $Q$.
We split the integral into the contributions from $I_n\cup \bar{A}_n$ and $O_n$. For $x\in A_n$ with $n\geq 1$ we have $\tilde{\eta}=1$ on $I_n\cup \bar{A}_n$ and can hence estimate, using \eqref{ass later t} and that $|z|\lesssim |x|$ on that region \begin{align}
\mel\left|\int_{I_n\cup \bar{A}_n}\nabla_z(\log|x-z|-\log|z|+\frac{x\cdot z}{|z|^2})\left(v_t^1(z)-v_t^1(0)-v_t^2(z)+v_t^2(0)\right)\left(\varpi_t^2(z)-\tilde{\eta}(z)\varpi_t^2(0)\right)\dd z\right|\nonumber\\
&\lesssim \int_{I_n\cup \bar{A}_n} \left(\frac{1}{|x-z|}+\frac{1}{|z|}+\frac{|x|}{|z|^{2}}\right)\left|v_t^1(z)-v_t^2(z)\right||z|^\alpha\dd z\nonumber\\
&\quad+|v_t^1(0)-v_t^2(0)|\int_{I_n\cup \bar{A}_n} \left(\frac{1}{|x-z|}+\frac{1}{|z|}+\frac{|x|}{|z|^{2}}\right)|z|^\alpha\dd z\nonumber\\
&\begin{aligned}&\lesssim |x|^\alpha\int_{I_n\cup \bar{A}_n} \frac{1}{|x-z|}\left|v_t^1(z)-v_t^2(z)\right|\dd z+|x|^{\alpha-1}\int_{I_n\cup \bar{A}_n}|v_t^1(z)-v_t^2(z)|\dd z\\
&\quad +\left|v_t^1(0)-v_t^2(0)\right|\int_{I_n\cup \bar{A}_n} \left(\frac{1}{|x-z|}+\frac{1}{|z|}+\frac{|x|}{|z|^{2}}\right)|z|^\alpha\dd z.\label{est q}\end{aligned}
\end{align}
%
%
By e.g.\ Young's convolution inequality, we see that \begin{align*}
\mel\norm{|x|^{\alpha-2+\frac{\beta}{2}}\int_{I_n\cup \bar{A}_n} \frac{1}{|x-z|}|v_t^1(z)-v_t^2(z)|\dd z}_{L_x^2(A_n)}\lesssim 2^{-(\alpha-2+\beta/2)n}\norm{v_t^1-v_t^2}_{L^2(I_n\cup \bar{A}_n)}\\
&\lesssim 2^{-(\alpha-2+\frac{\beta}{2})n}\norm{v_t^1-v_t^2}_{L^2(B_2(0))}\\
&\lesssim 2^{-(\alpha-2-\frac{\beta}{2})n}\sqrt{E_t}
\end{align*}
where the last step used Lemma \ref{l2 lemma} and \eqref{D m}.
By summing up the square of this over all dyadic annuli we see that \begin{align}
\norm{\sum_{m\geq 1}\mathds{1}_{A_m}(x)|x|^{\alpha-2+\frac{\beta}{2}}\int_{I_n\cup \bar{A}_n} \frac{1}{|x-z|}|v_t^1(z)-v_t^2(z)|\dd z}_{L^2(\R^2)}\lesssim \sqrt{E_t}\label{1st q}
\end{align}

\noindent Similarly, we may estimate the second integral in \eqref{est q} by \begin{align}
|x|^{\alpha-1}\int_{I_n\cup \bar{A}_n}|v_t^1(z)-v_t^2(z)|\dd z\lesssim |x|^{\alpha-1}\sqrt{E_t}\label{2nd q}
\end{align}
by Lemma \ref{l2 lemma} and \eqref{D m}, which has the desired weighted $L^2$ bound by integrating in $x$.

The third integral in \eqref{est q} can be estimated by $E_t$ through e.g.\ \eqref{vt 0} as \begin{align}\begin{aligned}\label{3rd q}
   \mel|v_t^1(0)-v_t^2(0)|\int_{I_n\cup \bar{A}_n} \left(\frac{1}{|x-z|}+\frac{1}{|z|}+\frac{|x|}{|z|^{2}}\right)|z|^\alpha\dd z\lesssim \sqrt{E_t} \int_{I_n\cup \bar{A}_n} \left(\frac{1}{|x-z|}+\frac{1}{|z|}+\frac{|x|}{|z|^{2}}\right)|z|^\alpha\dd z\\
   &\lesssim \sqrt{E_t}|x|^{\alpha+1},
\end{aligned}\end{align}
since $I_n\cup \bar{A}_n$ is a ball of radius $\approx |x|$.
This has the desired $L^2$-bound after integrating in $x$.

In the integral over the outer region $O_n$ we note that $|\varpi_t^2(z)-\tilde{\eta}(z)\varpi_t^2(0)|\lesssim |z|^{\alpha}$ by \eqref{ass later t} and because $\tilde{\eta}=1$ near $0$, and therefore, for $x\in A_n$, we have  \begin{align*}
\mel\left|\int_{O_n}\nabla_z\left(\log|x-z|-\log|z|+\frac{x\cdot z}{|z|^2}\right)\left(v_t^1(z)-v_t^1(0)-v_t^2(z)+v_t^2(0)\right)(\varpi_t^2(z)-\tilde{\eta}(z)\varpi_t^2(0))\dd z\right|\\
&\lesssim \int_{O_n} \left|\frac{(x-z)}{|x-z|^2}-\frac{z}{|z|^2}+\frac{x|z|^2-(x\cdot z)z}{|z|^4}\right||v_t^1(z)-v_t^2(z)||z|^\alpha\dd z\\
&\qquad+|v_t^1(0)-v_t^2(0)|\int_{O_n} \left|\frac{(x-z)}{|x-z|^2}-\frac{z}{|z|^2}+\frac{x|z|^2-(x\cdot z)z}{|z|^4}\right||z|^\alpha\dd z.
\end{align*}
Since $|\frac{(x-z)}{|x-z|^2}-\frac{z}{|z|^2}+\frac{x|z|^2-(x\cdot z)z}{|z|^4}|\lesssim |x||z|^{-2}$ for $|z|\geq 2|x|$, we see that for $x\in A_n$ we have  \begin{align}\begin{aligned}
    \mel\left|\int_{O_n}\nabla_z(\log|x-z|-\log|z|+\frac{z\cdot x}{|z|^2})\left(v_t^1(z)-v_t^1(0)-v_t^2(z)+v_t^2(0)\right)(\varpi_t^2(z)-\varpi_t^2(0))\dd z\right|\\
    &\lesssim |x|\int_{O_n} |v_t^1(z)-v_t^2(z)|\dd z+|v_t^1(0)-v_t^2(0)|\int_{O_n} |x|\dd z\lesssim |x|\sqrt{E_t}\label{4th q}
\end{aligned}\end{align}
where the last step uses Lemma \ref{l2 lemma} and \eqref{D m} for the first integral and \eqref{vt 0} for the second difference, as well as $|x|, |z|\lesssim 1$.

Putting \eqref{1st q}, \eqref{2nd q}, \eqref{3rd q} and \eqref{4th q} together, we see that \begin{align}
    \norm{|\cdot|^{\beta/2-2}Q}_{L^2(B_\frac{1}{2}(0))}\lesssim \sqrt{E_t}.\label{fin est q}
\end{align}

\subsubsection{The term $R$.} 
We partially integrate to see that \begin{align*}
R(x)=|\varpi_0(0)|\left|\int_{\R^2}(\log|x-z|-\log|z|+\frac{x\cdot z}{|z|^2})\left(\bar u_t^1(z)-\bar u_t^2(z)\right)\nabla\tilde{\eta}(z)\dd z\right|.
\end{align*}
Note that by the fundamental theorem of calculus we have \begin{align*}
    \left|\log|x-z|-\log|z|+\frac{x\cdot z}{|z|^2}\right|\lesssim |x|^2
\end{align*}
for $|z|\in (1,2)$ and $|x|\leq \frac{1}{2}$, which in particular contains the support of $\nabla \tilde{\eta}$. Therefore, for $|x|\leq \frac{1}{2}$, we have \begin{align*}
    R(x)\lesssim \int_{\R^2} |x|^2\left|\bar u_t^1(z)-\bar u_t^2(z)\right||\nabla\tilde{\eta}(z)|\dd z\lesssim |x|^2\left(\norm{v_t^1-v_t^2}_{L^2(B_2(0)}+|v_t^1(0)-v_t^2(0)|\right).
\end{align*}
By Lemma \ref{l2 lemma}, \eqref{D m} and \eqref{vt 0} we therefore see that \begin{align*}
    R(x)\lesssim  |x|^2\sqrt{E_t},
\end{align*}
and hence also that \begin{align*}
\norm{|\cdot|^{\beta/2-2}R}_{L^2(B_{\frac{1}{2}}(0))}\lesssim \sqrt{E_t}.
\end{align*}
Combining this estimate and \eqref{fin est p'}, \eqref{fin est p''} and  \eqref{fin est q} again shows \eqref{deri dpsi1}.

\subsubsection{Proof of \eqref{deri dpsi2}.}
Note that by definition it holds that $\Psi_R^1(t,0)-\Psi_R^2(t,0)=0$ and \begin{align*}\nabla_x^\perp (\Psi_R^1(t,x)-\Psi_R^2(t,x))=v_t^1(x)-v_t^1(0)-v_t^2(x)-v_t^2(0).\end{align*}
We may hence use Hardy's inequality \cite[Formula (1.3.3)]{mayza} to estimate \begin{align*}
    \int_{B_\frac{1}{2}(0)}\left|\Psi_R^1(t,x)-\Psi_R^2(t,x)\right|^2|x|^{\beta-4}\dx\lesssim \int_{B_\frac{1}{2}(0)}\left|v_t^1(x)-v_t^1(0)-v_t^2(x)-v_t^2(0)\right|^2|x|^{\beta-2}\dx\lesssim E_t,
\end{align*}
where the last step used Lemma \ref{weighted l2}.\hfill\qedsymbol

\subsection{Proof of Lemma \ref{det d}}
Let us first note that \eqref{det d 4} is a direct consequence of \eqref{det d 1}-\eqref{det d 3}, the identity \eqref{psir id}, \eqref{linfty psir} and the discrete product rule.

Indeed, using these, we have \begin{align*}
\mel \left|(\de_t d^1)(t,x)-(\de_t d^1)(t,y)\right|=\left||x|e^{2\pi \Psi_R^1(t,x)}\de_t\Psi_R^1(t,x)-|y|e^{2\pi \Psi_R^1(t,y)}\de_t\Psi_R^1(t,y)\right|\\
&\lesssim \left||x|-|y|\right|\left|e^{2\pi \Psi_R^1(t,x)}\de_t\Psi_R^1(t,x)\right|+|y|\left|e^{2\pi \Psi_R^1(t,x)}-e^{2\pi \Psi_R^1(t,y)}\right|\left|\de_t\Psi_R^1(t,x)\right|\\
&\quad+|y|\left|\de_t\Psi_R^1(t,x)-\de_t\Psi_R^1(t,y)\right|\\
&\lesssim |x|^2|x-y|\log\left(2+\frac{1}{|x-y|}\right),
\end{align*}
which also used the assumption that $|x|\approx |y|$.

To prove the estimates \eqref{det d 1}-\eqref{d est 3}, we fix an $x\in A_n$ and again use the regions $I_n$, $\bar{A}_n$ and $O_n$, as defined in \eqref{def regions}.

\subsubsection{The estimate \eqref{det d 1}} 
We use the explicit form of the Biot-Savart law, the governing equations and partial integration to see that \begin{align}\begin{aligned}
\mel 2\pi \dfrac\Psi_R^1(t,x)=\dfrac \int_{\R^2} \left(\log|x-z|-\log|z|+\frac{z\cdot x}{|z|^2}\right)\varpi_t^1(z)\dd z\\
&=\int_{B_{S_2}(0)} \bar u_t^1(z)\cdot\nabla_z\left(\log|x-z|-\log|z|+\frac{z\cdot x}{|z|^2}\right)\left(\varpi_t^1(z)-\tilde{\eta}(z)\varpi_t^1(0)\right)\dd z\\
&\quad -\int_{\R^2} \left(\log|x-z|-\log|z|+\frac{z\cdot x}{|z|^2}\right)\bar u_t^1(z)\cdot\nabla_z\tilde{\eta}(z)\varpi_t^1(0)\dd z
\label{det psir1}\end{aligned}\end{align}
here $\tilde{\eta}$ is again a smooth, radially symmetric and nonnegative cutoff function, not depending on the other quantities, which equals $1$ on $B_1(0)$ and which is supported in $B_2(0)$.

The second integral here is trivially $\lesssim |x|^2$ since $|\log|x-z|-\log|z|+\frac{z\cdot x}{|z|^2}|\lesssim |x|^2$ and $|\bar u_t^1(z)|\lesssim 1$ for $z\in \supp \nabla \tilde{\eta}$ by \eqref{uni bd u} and the assumption that $|x|\leq S_1\leq \frac{1}{2}$.

We therefore only need to estimate the first integral in \eqref{det psir1}.
Over the region $I_n$, we estimate all terms directly with the triangle inequality and, using that the point vortex velocity is orthogonal to $z$ and that $\tilde{\eta}=1$ here, we see that \begin{align}
\mel\left|\int_{I_n} \bar u_t^1(z)\cdot\nabla_z\left(\log|x-z|-\log|z|+\frac{z\cdot x}{|z|^2}\right)\left(\varpi_t^1(z)-\tilde{\eta}(z)\varpi_t^1(0)\right)\dd z\right|\nonumber\\
&=\biggl|\int_{I_n}\left( \frac{(x-z)\cdot \bar u_t^1(z)}{|x-z|^2}-\frac{z\cdot (v_t^1(z)-v_t(0))}{|z|^2}+\frac{\bar u_t^1(z)\cdot x}{|z|^2}-\frac{(z\cdot x)(v_t^1(z)-v_t(0))\cdot z}{2|z|^4}\right)\nonumber\\
&\quad\times \left(\varpi_t^1(z)-\varpi_t^1(0)\right)\dd z\biggr|\nonumber\\
&\lesssim \int_{I_n} \left(|x|^{-1}|\bar u_t^1(z)|+|v_t^1(z)-v_t^1(0)||z|^{-1}+|\bar u_t^1(z)||x||z|^{-2}+|v_t^1(z)-v_t^1(0)||x||z|^{-2}\right)|z|^\alpha\dd z\nonumber\\
&\lesssim \int_{I_n} |x|^{-1}|z|^{\alpha-1}+|z|^{\alpha}+|x||z|^{\alpha-3}\dd z\lesssim |x|^2\label{in}
\end{align} 
here we have also used that $|x-z|\approx |x|$ on $I_n$ and the estimates \eqref{est v} and \eqref{uni bd u} on the velocity, as well as \eqref{ass later t} and that $I_n$ is a ball of radius $\approx |x|$ by definition.

Over the region $\bar{A}_n$, we directly estimate everything with the triangle inequality to see that\begin{align}
\mel\left|\int_{\bar{A}_n}\nabla_z \left(\log|x-z|-\log|z|+\frac{x\cdot z}{|z|^2}\right)\cdot \bar u_t^1(z)(\varpi_t^1(z)-\tilde{\eta}(z)\varpi_t^1(0))\dd z\right|\nonumber\\
&\lesssim \int_{\bar{A}_n}\left(\frac{1}{|x-z|}+\frac{1}{|z|}+\frac{|x|}{|z|^2}\right)|\bar u_t^1(z)||z|^\alpha\dd z\nonumber\\
&\lesssim \int_{\bar{A}_n} \frac{|x|^{\alpha-1}}{|x-z|}+|x|^{\alpha-2}\dd z\nonumber\\
&\lesssim |x|^2+|x|\int_{B_{10|x|}(0)}\frac{1}{|z|}\dd z\lesssim |x|^2\label{an}
\end{align}
by the estimate \eqref{uni bd u} on $\bar u_t^1$ and the fact that $|x|\approx |z|$ here.

Over the third region $O_n$, we use the elementary estimate \begin{align*}
\left|\nabla_z\left(\log|x-z|-\log|z|+\frac{z\cdot x}{|z|^2}\right)\right|\lesssim |x|^2|z|^{-3}, 
\end{align*}
which holds for all $z$ and $x$ with $|z|\geq 2|x|$ (so in particular for all $z$ in $O_n$) and can be easily shown with the fundamental theorem of calculus.
Together with \eqref{uni bd u} and \eqref{ass later t}, this yields that \begin{align}
    \left|\int_{O_n}\bar u_t^1(z)\cdot\nabla_z\left(\log|x-z|-\log|z|+\frac{z\cdot x}{|z|^2}\right)\left(\varpi_t^1(z)-\tilde{\eta}(z)\varpi_t^1(0)\right)\dd z\right|\lesssim \int_{O_n}|x|^2 |z|^{\alpha-4}\dd z \lesssim |x|^2.\label{on}
\end{align}
Putting \eqref{in}, \eqref{an} and \eqref{on} together shows \eqref{det d 1}.

\subsubsection{The estimate \eqref{det d 2}}
We may rewrite \begin{align*}
\mel 2\pi(\Psi_R^1(t,x)-\Psi_R^1(t,y))=\int_{\R^2} \left(\log|x-z|-\log|y-z|+\frac{z\cdot (x-y)}{|z|^2}\right)\varpi_t^1(z)\dd z\\
&=\int_{B_{S_2}(0)} \left(\log|x-z|-\log|y-z|+\frac{z\cdot (x-y)}{|z|^2}\right)(\varpi_t^1(z)-\tilde{\eta}(z)\varpi_t^1(0))\dd z\\
&\quad+\int_{\R^2} \left(\log|x-z|-\log|y-z|\right)\tilde{\eta}(z)\varpi_t^1(0))\dd z
\end{align*}
where the last step used the antisymmetry of $\frac{z\cdot (x-y)}{|z|^2}$.

We first analyze the second integral. It can be estimated by $|x-y|\max(|x|,|y|)$ since $\tilde{\eta}*\log|\cdot|$ is a $C^2$ function near $0$ with a critical point at $0$ (since it is radially symmetric).

We split the first integral into the three regions $I_n,\bar{A}_n, O_n$ again. Let us note that by the assumption on $y$, we have $y\in \bar{A}_n$ and $\dist(y,\de\bar{A}_n)\gtrsim |x|$.

We have \begin{align}
    \mel\left|\int_{I_n} \left(\log|x-z|-\log|y-z|+\frac{z\cdot (x-y)}{|z|^2}\right)(\varpi_t^1(z)-\tilde{\eta}\varpi_t^1(0))\dd z\right|\lesssim\int_{I_n} \frac{|x-y|}{|z|}|z|^\alpha\dd z\nonumber\\
    &\lesssim |x-y|\max(|x|,|y|)^3\label{in2}
\end{align}
which used that $|x-z|\approx |y-z|\approx |x|\gtrsim |z|$ on this region, as well as \eqref{ass later t}.

On the region $\bar{A}_n$ we use that $|\log|x-z|-\log|y-z||\lesssim |x-y||z|^{-1}$ for $x,y,z\in \bar{A}_n$, as one can see by estimating the Lipschitz constant in $\bar{A}_n$ with the supremum of the derivative.

This yields by \eqref{ass later t} that \begin{align}
   \mel \left|\int_{\bar{A}_n} \left(\log|x-z|-\log|y-z|+\frac{z\cdot (x-y)}{|z|^2}\right)\left(\varpi_t^1(z)-\tilde{\eta}(z)\varpi_t^1(0)\right)\dd z\right|\nonumber\\
    &\lesssim \int_{\bar{A}_n} |x-y||z|^{\alpha-1}\dd z\lesssim |x-y|\max(|x|,|y|)^3.\label{an2}
\end{align}

\noindent We move on to the region $O_n$. Here we estimate the first bracket with the fundamental theorem of calculus as \begin{align}\begin{aligned}
\mel \left|\log|x-z|-\log|y-z|+\frac{z\cdot (x-y)}{|z|^2}\right|=\left|\int_0^1\left(\frac{(1-s)x+sy-z}{|(1-s)x+sy-z|^2}+\frac{z}{|z|^2}\right)\cdot(x-y) \dd s\right|\\
&\lesssim (|x|+|y|)|z|^{-2}|x-y|.\label{fc}
\end{aligned}\end{align}
This yields that \begin{align}
&\left|\int_{O_n} \left(\log|x-z|-\log|y-z|+\frac{z\cdot (x-y)}{|z|^2}\right)(\varpi_t^1(z)-\tilde{\eta}(z)\varpi_t^1(0))\dd z\right|\nonumber\\
&\lesssim  (|x|+|y|)|x-y|\int_{O_n}|z|^{\alpha-2}\dd z\lesssim  \max(|x|,|y|)|x-y|.\label{on2}
\end{align}
Putting \eqref{in2}, \eqref{an2} and \eqref{on2} together yields \eqref{det d 2}.

\subsubsection{The estimate \eqref{det d 3}}

We use the governing equations and partial integration to rewrite \begin{align*}
\mel 2\pi\left(\de_t \Psi_R^1(t,x)-\de_t \Psi_R^1(t,y)\right)\\
&=\int_{B_{S_2}(0)} \bar u_t^1(z)\cdot\nabla_z\left(\log|x-z|-\log|y-z|+\frac{z\cdot (x-y)}{|z|^2}\right) (\varpi_t^1(z)-\tilde{\eta}(z)\varpi_t^1(0))\dd z\\
&\quad+\int_{B_{S_2}(0)} \left(\log|x-z|-\log|y-z|+\frac{z\cdot (x-y)}{|z|^2}\right) \bar u_t^1(z)\cdot\nabla\tilde{\eta}(z)\varpi_t^1(0)\dd z.
\end{align*}
The second integral is again bounded by $|x-y|\max(|x|,|y|)$ since the first bracket has this bound on $\supp\nabla \tilde{\eta}$.

This first integral is split into the three regions $I_n, \bar{A}_n, O_n$ yet again. We have \begin{align}
 \mel   \left|\int_{I_n} \bar u_t^1(z)\cdot\nabla_z\left(\log|x-z|-\log|y-z|+\frac{z\cdot (x-y)}{|z|^2}\right) \left(\varpi_t^1(z)-\tilde{\eta}(z)\varpi_t^1(0)\right)\dd z\right|\nonumber\\
 &\lesssim   \int_{I_n} \left|\frac{x-z}{|x-z|^2}-\frac{y-z}{|y-z|^2}+\nabla_z\frac{z\cdot (x-y)}{|z|^2}\right| |z|^{\alpha-1}\dd z\nonumber\\
&    \lesssim \int_{I_n} |x-y||z|^{\alpha-3}\dd z\nonumber\\
&\lesssim |x-y|\max(|x|,|y|)\nonumber
\end{align}
by \eqref{uni bd u}, \eqref{ass later t} and because $|x-z|\approx |y-z|\approx |x|\gtrsim |z|$ on this region.

On the region $\bar{A}_n$, we  split the bracket to see that \begin{align*}
     \mel\left|\int_{\bar{A}_n} \nabla_z\left(\log|x-z|-\log|y-z|+\frac{z\cdot (x-y)}{|z|^2}\right)\bar u_t^1(z)(\varpi_t^1(z)-\tilde{\eta}(z)\varpi_t^1(0))\dd z\right|\\
     &\lesssim |h(x)-h(y)|+|x-y|\int_{\bar{A}_n} \frac{1}{|z|^2}|\bar u_t^1(z)|\left|\varpi_t^1(z)-\varpi_t^1(0)\right|\dd z\\
     &\lesssim |h(x)-h(y)|+|x-y|\int_{\bar{A}_n} |z|^{\alpha-3}\dd z,
\end{align*}
where $h$ is a shorthand for $((\bar u_t^1(\cdot)(\varpi_t^1(\cdot)-\varpi_t^1(0))\mathds{1}_{\bar{A}_n})*\log|\cdot|$. To estimate this, we note that, thanks to Lemma \ref{ll lemma app}, $h$ is log-Lipschitz with modulus $\lesssim \norm{\bar u_t^1(\cdot)(\varpi_t^1(\cdot)-\varpi_t^1(0))\mathds{1}_{\bar{A}_n}}_{L^\infty}$. This norm in turn is $\lesssim \max(|x|,|y|)$ by \eqref{ass later t} and \eqref{uni bd u}. Hence we see that \begin{align*}
    \mel\left|\int_{\bar{A}_n} \nabla\left(\log|x-z|-\log|y-z|+\frac{z\cdot (x-y)}{|z|^2}\right)\bar u_t^1(z)(\varpi_t^1(z)-\varpi_t^1(0))\dd z\right|\\
    &\lesssim \max(|x|,|y|)|x-y|\log\left(2+\frac{1}{|x-y|}\right).
\end{align*}
On the region $O_n$, we estimate with the fundamental theorem of calculus, similarly as in \eqref{fc} above \begin{align*}
    \left|\nabla_z\left(\log|x-z|-\log|y-z|+\frac{z\cdot (x-y)}{|z|^2}\right)\right|\lesssim (|x|+|y|)|z|^{-3}|x-y|,
\end{align*}
which holds whenever $|z|\geq 2\max(|x|,|y|)$, as it is the case on $O_n$.

This yields that \begin{align*}
\mel\left|\int_{O_n} \nabla_z\left(\log|x-z|-\log|y-z|+\frac{z\cdot (x-y)}{|z|^2}\right)\bar u_t^1(z)(\varpi_t^1(z)-\tilde{\eta}(z)\varpi_t^1(0))\dd z\right|\\
&\lesssim (|x|+|y|)|x-y|\int_{O_n}|z|^{\alpha-4}\dd z\\
&\lesssim  \max(|x|,|y|)|x-y|.
\end{align*}
Combining these estimates shows \eqref{det d 3}.

\hfill\qedsymbol

\subsection{Proof of Lemma \ref{weighted l2}}
We begin with \eqref{weighted l2 est}.
We again use the dyadic annuli $A_n$ and balls $I_n$ as defined in \eqref{def regions}, and split \begin{align*}
   \mel \int_{B_1(0)} |x|^{\beta-2}|v_t^1(x)-v_t^2(x)|^2\dx\\
   &\leq \sum_{n=0}^\infty \int_{A_n} |x|^{\beta-2}\bigl|\BS\Bigr[(\mathds{1}_{I_{n-4}}(\varpi_0-\varpi_0(0))\circ (\mathcal{X}_t^1)^{-1}-(\mathds{1}_{I_{n-4}}(\varpi_0-\varpi_0(0)))\circ (\mathcal{X}_t^2)^{-1}\Bigr]\Bigr|^2\dx\\
   &\quad+\int_{A_n} |x|^{\beta-2}\Bigl|\BS\Bigl[(\mathds{1}_{B_{\frac{1}{2}S_2}(0)\backslash I_{n-4}}(\varpi_0-\varpi_0(0)))\circ (\mathcal{X}_t^1)^{-1}\\
   &\qquad-(\mathds{1}_{B_{\frac{1}{2}S_2}(0)\backslash I_{n-4}}(\varpi_0-\varpi_0(0)))\circ (\mathcal{X}_t^2)^{-1}\Bigr]\Bigr|^2\dx\\
   &\quad+\int_{B_{1}(0)} |x|^{\beta-2}\bigl|\BS\bigr[(\mathds{1}_{B_{\frac{1}{2}S_2}(0)}\varpi_0(0))\circ (\mathcal{X}_t^1)^{-1}-(\mathds{1}_{B_{\frac{1}{2}S_2}(0)}\varpi_0(0))\circ (\mathcal{X}_t^2)^{-1}\bigl]\bigr|^2\dx\\
    &=:I+II+III,
\end{align*}
where we have used that $\varpi_0$ is supported in $B_{\frac{1}{2}S_2}(0)$ by definition. We estimate each term individually.

For the term $I$, we use Lemma \ref{l2 lemma} and that $\norm{\mathds{1}_{I_{n-4}}(\varpi_0-\varpi_0(0))}_{L^\infty}\lesssim 2^{-n\alpha}$ by \eqref{ass later t}, yielding the bound \begin{align*}
    I\lesssim \sum_{n=1}^\infty 2^{2n(2-\alpha-\beta)}\int_{B_2(0)} \left|(\mathcal{X}_t^1)(x)-(\mathcal{X}_t^2)(x)\right|^2\dx\lesssim \int_{B_2(0)} \left|(\mathcal{X}_t^1)(x)-(\mathcal{X}_t^2)(x)\right|^2\dx
\end{align*}
which also used that $|x|^{\beta-2}\approx 2^{-(2-\beta)n}$ on each $A_n$ and that the preimage of $B_1(0)$ is contained in $B_2(0)$ by \eqref{mod1}.

This is bounded by $E_t$,  thanks to \eqref{D m}.

The term $II$ on the other hand uses the explicit form of the Biot-Savart law.
It holds that\begin{align*}
   \mel\left|\BS\left[(\mathds{1}_{B_{\frac{1}{2}S_2(0)}\backslash I_{n-4}}(\varpi_0-\varpi_0(0)))\circ (\mathcal{X}_t^1)^{-1}-\mathds{1}_{B_{\frac{1}{2}S_2(0)}\backslash I_{n-4}}(\varpi_0-\varpi_0(0)))\circ (\mathcal{X}_t^2)^{-1}\right]\right|(x)\\
   &\lesssim \int_{B_{\frac{1}{2}S_2(0)}\backslash I_{n-4}} \left|\frac{x-\mathcal{X}_t^1(z)}{|x-\mathcal{X}_t^1(z)|^2}-\frac{x-\mathcal{X}_t^2(z)}{|x-\mathcal{X}_t^2(z)|^2}\right| \left| \varpi_0(z)-\varpi_0(0)\right|  \dd z\\
   &\lesssim \int_{B_{S_2}(0)\backslash I_{n-4}} \frac{|\mathcal{X}_t^1(z)-\mathcal{X}_t^2(z)|}{|z|^2}|z|^\alpha \dd z, 
\end{align*}
where we have used that for $z\in \R^2\backslash I_{n-4}$, it holds that $|x-\mathcal{X}_t^1(z)|\approx |x-\mathcal{X}_t^2(z)|\approx |z|$ by \eqref{mod1}
and also used \eqref{ass later t}. By Cauchy-Schwarz and \eqref{D m}, this can now be estimated by \begin{align*}
    \int_{B_{S_2}(0)\backslash I_{n-4}} \frac{|\mathcal{X}_t^1(z)-\mathcal{X}_t^2(z)|}{|z|^2}|z|^\alpha \dd z  \lesssim \left(\int_{B_{S_2}(0)}k_t(z)^2 z^{\beta-2}\dd z \right)^\frac{1}{2}\left(\int_{B_{S_2}(0)}|z|^{2\alpha-2+\beta}\dd z\right)^\frac{1}{2}\lesssim \sqrt{E_t}.
\end{align*}
Using this to estimate $II$, we see that \begin{align*}
II\lesssim \sum_{n=1}^\infty E_t \int_{A_n} |x|^{\beta-2}\dd x\lesssim E_t
\end{align*}
For the term $III$, we note that $((\mathds{1}_{B_{\frac{1}{2}S_2}(0)}\varpi_0(0))\circ (\mathcal{X}_t^1)^{-1}-(\mathds{1}_{B_{\frac{1}{2}S_2}(0)}\varpi_0(0))\circ (\mathcal{X}_t^2)^{-1})(z)=0$, unless $(\mathcal{X}_t^1)^{-1}(z)$ lies in $B_{\frac{1}{2}S_2}(0)$ and $(\mathcal{X}_t^2)^{-1}(z)$ does not (or vice versa). By \eqref{mod1}, this can only happen for  $x\in B_{\frac{3}{4}S_2}(0)\backslash B_{\frac{1}{4}S_2}(0)$. 
We can therefore take a $C^1$-function $h_x$ which agrees with $\frac{x-z}{|x-z|^2}$ outside of $B_{\frac{1}{4}S_2}(0)$ and whose $C^1$-norm is bounded uniformly in $x\in B_1(0)$, 
and write 
\begin{align*}
\mel \left|\BS\left[(\mathds{1}_{B_{\frac{1}{2}S_2}(0)}\varpi_0(0))\circ (\mathcal{X}_t^1)^{-1}-(\mathds{1}_{B_{\frac{1}{2}S_2}(0)}\varpi_0(0))\circ (\mathcal{X}_t^2)^{-1}\right](x)\right|\\
& =\left|\int_{\R^2}h_x(z) \left(\mathds{1}_{B_{\frac{1}{2}S_2}(0)}\varpi_0(0))\circ (\mathcal{X}_t^1)^{-1}-(\mathds{1}_{B_{\frac{1}{2}S_2}(0)}\varpi_0(0))\circ (\mathcal{X}_t^2)^{-1}\right)(z)\dd z\right|\\
&\leq \int_{B_{S_2}(0)}\left|h_x(\mathcal{X}_t^1(z))-h_x(\mathcal{X}_t^2(z))\right||\varpi_0(0)| \dd z\\
&\lesssim \int_{B_{S_2}(0)} \left|\mathcal{X}_t^1(z)-\mathcal{X}_t^2(z)\right|\dd z\lesssim \sqrt{E_t}.
\end{align*}
Since $|x|^{\beta-2}$ is integrable, this immediately yields the desired bound on $III$.

\subsubsection{The estimate \eqref{vt 0}.}
We compute \begin{align}\begin{aligned}
  \mel 2\pi \left|v_t^1(0)-v_t^2(0)\right|=\left|\int_{\R^2} \frac{x^\perp}{|x^2|}\left(\varpi_t^1(x)-\varpi_t^2(x)\right)\dd x\right|\\
 & \leq \int_{\R^2} |\varpi_0(x)-\mathds{1}_{B_{\frac{1}{2}S_2}(0)}\varpi_0(0)|\left|\frac{\Chi_t^1(x)}{|\Chi_t^1(x)|^2}-\frac{\Chi_t^2(x)}{|\Chi_t^2(x)|^2}\right|\dd x\\
 &\quad +\left|\int_{\R^2}\frac{x^\perp}{|x|^2}\left(\mathds{1}_{B_{\frac{1}{2}S_2}(0)}\varpi_0(0)\circ (\mathcal{X}_t^1)^{-1}(x)-\mathds{1}_{B_{\frac{1}{2}S_2}(0)}\varpi_0(0)\circ (\mathcal{X}_t^2)^{-1}(x)\right)\right|\label{est vt0}
\end{aligned}\end{align}
We can estimate the first integral here as 
\begin{align*} 
\mel \int_{\R^2} \left|\varpi_0(x)-\mathds{1}_{B_{\frac{1}{2}S_2}(0)}\varpi_0(0)\right|\left|\frac{\Chi_t^1(x)}{|\Chi_t^1(x)|^2}-\frac{\Chi_t^2(x)}{|\Chi_t^2(x)|^2}\right|\dd x\\
&\lesssim\int_{\R^2}  \left|\varpi_0(x)-\mathds{1}_{B_{\frac{1}{2}S_2}(0)}\varpi_0(0)\right||x|^{-2}|\Chi_t^1(x)-\Chi_t^2(x)|\dd x\\
&\lesssim \int_{B_{S_2}(0)} |\Chi_t^1(x)-\Chi_t^2(x)||x|^{\alpha-2}\dd x\lesssim \sqrt{E_t}
\end{align*}
where we have used \eqref{mod1}, Cauchy-Schwarz and \eqref{D m}.

For the second integral in \eqref{est vt0}, we argue similarly as for above and use that by \eqref{mod1}, the first difference can only be non-zero if $|x|\geq \frac{1}{4}$ and that we can again replace $\frac{x}{|x|^2}$ by a $C^1$-function $h$ (not depending on any of the other quantities) which agrees with $\frac{x}{|x|^2}$ on this region to see that \begin{align*}
    \mel\left|\int_{\R^2}\frac{x^\perp}{|x|^2}\left(\mathds{1}_{B_{\frac{1}{2}S_2}(0)}\varpi_0(0)\circ (\mathcal{X}_t^1)^{-1}(x)-\mathds{1}_{B_{\frac{1}{2}S_2}(0)}\varpi_0(0)\circ (\mathcal{X}_t^2)^{-1}(x)\right)\dd x\right|\\
&=\left|\int_{B_{\frac{1}{2}S_2}(0)}h(\mathcal{X}_t^1(x))-h(\mathcal{X}_t^2(x))\dd x\right|\\
&\lesssim \norm{h}_{C^1}\int_{B_{\frac{1}{2}S_2}(0)}\left|\mathcal{X}_t^1(x))-\mathcal{X}_t^2(x)\right|\dd x\lesssim \sqrt{E_t},
\end{align*}
where the last step uses Cauchy-Schwarz and \eqref{D m}. Combining the estimates yields the statement.\hfill\qedsymbol

\textbf{Statement on competing interests:} The author declares that he has no competing interests.

\textbf{Statement on AI usage:} Generative AI (ChatGPT 5.6 Luna and Claude Sonnet 5) was used solely for proofreading.
No part of this work was written by AI.
All scientific content was developed by the author alone, who takes full responsibility for this work.

\appendix
\section{Classical estimates}

\begin{lemma}\label{l2 lemma}
Let $\Phi_1,\Phi_2:\R^2\rightarrow \R^2$ be volume-preserving, invertible and measureable. Let $f\in L^1(\R^2)\cap L^\infty(\R^2)$, then it holds that \begin{align*}
    \norm{\BS[f\circ \Phi_1]-\BS[f\circ \Phi_2]}_{L^2(\R^2)}\lesssim \norm{f}_{L^\infty(\R^2)}\norm{\Phi_1^{-1}-\Phi_2^{-1}}_{L^2(\supp f)}.
\end{align*}
\end{lemma}
\begin{proof}
We have that \begin{align*}
   \mel \norm{\BS[f\circ \Phi_1]-\BS[f\circ \Phi_2]}_{L^2(\R^2)}=\sup_{g\in C_c^\infty(\R^2),\norm{g}_{L^2}\leq 1}\scalar{g}{\BS[f\circ \Phi_1]-\BS[f\circ \Phi_2]}\\
   &=\sup_{g\in C_c^\infty(\R^2),\norm{g}_{L^2}\leq 1}\scalar{\BS[g]\circ \Phi_1^{-1}-\BS[g]\circ \Phi_2^{-1}}{f}.
\end{align*}
We now use that for any $H^1$ function $h$ it holds that $$|h(x)-h(y)|\lesssim |x-y|(M|\nabla h|(x)+M|\nabla h|(y)),$$ where $M$ is the classical Hardy-Littlewood maximal function, see e.g.\ \cite{aalto2009maximal} for more details. This yields that \begin{align*}
&\norm{\BS[f\circ \Phi_1]-\BS[f\circ \Phi_2]}_{L^2(\R^2)}\\
&\lesssim \sup_{g\in C_c^\infty(\R^2),\norm{g}_{L^2}\leq 1}\int_{\R^2} \left|\Phi_1^{-1}(x)- \Phi_2^{-1}(x)\right| (M|\nabla \BS[g]|(\Phi_1^{-1}(x))+M|\nabla \BS[g]|(\Phi_2^{-1}(x))) |f(x)|\dd x\\
&\lesssim \norm{f}_{L^\infty(\R^2)}\norm{\Phi_1^{-1}-\Phi_2^{-1}}_{L^2(\supp f)}\sup_{g\in C_c^\infty(\R^2),\norm{g}_{L^2}\leq 1}\norm{M|\nabla\BS[g]|}_{L^2(\R^2)}.
\end{align*}
The last supremum is $\lesssim 1$ since $M$ maps $L^2$ to $L^2$ boundedly, see e.g.\ \cite[Thm.\ 2.1.6]{grafakos2008classical} and $\norm{\BS[g]}_{\dot{H}^1}=\norm{g}_{L^2}$ by Plancherel.
\end{proof}

\begin{lemma}\label{ll lemma app}
Let $f\in L^1(\R^2)\cap L^\infty(\R^2)$ and let $x,y\in \R^2$, then it holds that \begin{align*}
    \left|\BS[f](x)-\BS[f](y)\right|\lesssim |x-y|\log\left(2+\frac{1}{|x-y|}\right)(\norm{f}_{L^\infty(\R^2)}+\norm{f}_{L^1(\R^2)}).
\end{align*}
\end{lemma}
See e.g.\ \cite[Lemma 8.1]{MajdaBertozzi} for a proof.

\bibliographystyle{abbrv}
  \bibliography{refsVW}

@book{marchioro2012mathematical,
  title={Mathematical theory of incompressible nonviscous fluids},
  author={Marchioro, Carlo and Pulvirenti, Mario},
  volume={96},
  year={2012},
  publisher={Springer Science \& Business Media}
}

@incollection{marchioro1991vortex,
  title={{On the vortex--wave system}},
  author={Marchioro, Carlo and Pulvirenti, Mario},
  booktitle={Mechanics, analysis and geometry: 200 years after Lagrange},
  pages={79--95},
  year={1991},
  publisher={Elsevier}
}

@article{glass2014motion,
  title={{On the motion of a small body immersed in a two-dimensional incompressible perfect fluid}},
  author={Glass, Olivier and Lacave, Christophe and Sueur, Franck},
  journal={Bulletin de la soci{\'e}t{\'e} math{\'e}matique de France},
  volume={142},
  number={3},
  pages={489--536},
  year={2014}
}

@article{lacave2009uniqueness,
  title={{Uniqueness for the vortex-wave system when the vorticity is constant near the point vortex}},
  author={Lacave, Christophe and Miot, Evelyne},
  journal={SIAM journal on mathematical analysis},
  volume={41},
  number={3},
  pages={1138--1163},
  year={2009},
  publisher={SIAM}
}

@article{flandoli2026early,
  title={{The early stage of the motion along the gradient of a concentrated vortex structure}},
  author={Flandoli, Franco and Palmieri, Matteo and Viviani, Milo},
  journal={Physica D: Nonlinear Phenomena},
  pages={135289},
  year={2026},
  publisher={Elsevier}
}

@article{glass2019dynamics,
    AUTHOR = {Glass, Olivier and Sueur, Franck},
     TITLE = {{Dynamics of several rigid bodies in a two-dimensional ideal fluid and convergence to vortex systems}},
   JOURNAL = {Bull. Soc. Math. France},
  FJOURNAL = {Bulletin de la Soci\'et\'e{} Math\'ematique de France},
    VOLUME = {154},
      YEAR = {2026},
    NUMBER = {1},
     PAGES = {1--106},
      ISSN = {0037-9484,2102-622X},
   MRCLASS = {35Q31 (74F10 76B10)},
  MRNUMBER = {5074192},
}

@article{nguyen2019inviscid,
  title={{The inviscid limit of Navier--Stokes equations for vortex-wave data on $\mathbb{R}^2$}},
  author={Nguyen, Toan T and Nguyen, Trinh T},
  journal={SIAM Journal on Mathematical Analysis},
  volume={51},
  number={3},
  pages={2575--2598},
  year={2019},
  publisher={SIAM}
}

@article{bjorland2011vortex,
  title={{The vortex-wave equation with a single vortex as the limit of the Euler equation}},
  author={Bjorland, Clayton},
  journal={Communications in mathematical physics},
  volume={305},
  number={1},
  pages={131--151},
  year={2011},
  publisher={Springer}
}

@article{cobb2025existence,
  title={{Existence and uniqueness for the SQG vortex-wave system when the vorticity is constant near the point-vortex}},
  author={Cobb, Dimitri and Donati, Martin and Godard-Cadillac, Ludovic},
  journal={SIAM Journal on Mathematical Analysis},
  volume={57},
  number={3},
  pages={2316--2362},
  year={2025},
  publisher={SIAM}
}

@article{lopes2011existence,
  title={{Existence of a Weak Solution in $L^p$ to the Vortex-Wave System}},
  author={Lopes Filho, Milton C and Miot, Evelyne and Nussenzveig Lopes, Helena J},
  journal={Journal of nonlinear science},
  volume={21},
  number={5},
  pages={685--703},
  year={2011},
  publisher={Springer}
}

@article{agrawal2025uniqueness,
  title={{Uniqueness of the 2D Euler equation on rough domains}},
  author={Agrawal, Siddhant and Nahmod, Andrea R},
  journal={Journal of the European Mathematical Society},
  year={2025},
  publisher={EMS Press}
}

@article{agrawal2020uniqueness,
doi = {10.1088/1361-6544/ac586a},
url = {https://doi.org/10.1088/1361-6544/ac586a},
year = {2022},
month = {may},
publisher = {IOP Publishing},
volume = {35},
number = {6},
pages = {2767},
author = {Agrawal, Siddhant and Nahmod, Andrea R},
title = {{Uniqueness of the 2D Euler equation on a corner domain with non-constant vorticity around the corner}},
journal = {Nonlinearity}
}

@article{ionescu2022axi,
  title={{Axi-symmetrization near point vortex solutions for the 2D Euler equation}},
  author={Ionescu, Alexandru and Jia, Hao},
  journal={Communications on Pure and Applied Mathematics},
  volume={75},
  number={4},
  pages={818--891},
  year={2022},
  publisher={Wiley Online Library}
}

@article{gallagher2005uniqueness,
  title={{Uniqueness for the two-dimensional Navier--Stokes equation with a measure as initial vorticity}},
  author={Gallagher, Isabelle and Gallay, Thierry},
  journal={Mathematische Annalen},
  volume={332},
  number={2},
  pages={287--327},
  year={2005},
  publisher={Springer}
}

@article{pausader2024stability,
  title={{Stability of a point charge for the repulsive Vlasov--Poisson system}},
  author={Pausader, Beno{\^\i}t and Widmayer, Klaus and Yang, Jiaqi},
  journal={Journal of the European Mathematical Society},
  year={2024},
  publisher={EMS Press}
}

@article{marchioro2011cauchy,
  title={{The Cauchy problem for the 3-D Vlasov--Poisson system with point charges}},
  author={Marchioro, Carlo and Miot, Evelyne and Pulvirenti, Mario},
  journal={Archive for rational mechanics and analysis},
  volume={201},
  number={1},
  pages={1--26},
  year={2011},
  publisher={Springer}
}

@article{meyer2025long,
  title={Long time confinement of multiple concentrated vortices},
  author={Meyer, David},
  journal={arXiv preprint arXiv:2506.01477},
  year={2025}
}

@article{vishik2018a,
  title={{Instability and non-uniqueness in the Cauchy problem for the Euler equations of an ideal incompressible fluid. Part I}},
  author={Vishik, Misha},
  journal={arXiv preprint arXiv:1805.09426},
  year={2018}
}

@article{vishik2018b,
  title={{Instability and non-uniqueness in the Cauchy problem for the Euler equations of an ideal incompressible fluid. Part II}},
  author={Vishik, Misha},
  journal={arXiv preprint arXiv:1805.09440},
  year={2018}
}

@article{grotto2022burst,
  title={{Burst of point vortices and non-uniqueness of 2D Euler equations}},
  author={Grotto, Francesco and Pappalettera, Umberto},
  journal={Archive for Rational Mechanics and Analysis},
  volume={245},
  number={1},
  pages={89--125},
  year={2022},
  publisher={Springer}
}

@article{brue2026flexibility,
  title={{Flexibility of two-dimensional Euler flows with integrable vorticity}},
  author={Brue, Elia and Colombo, Maria and Kumar, Anuj},
  journal={Duke Mathematical Journal},
  volume={175},
  number={9},
  pages={1593--1647},
  year={2026},
  publisher={Duke University Press}
}

@incollection{aalto2009maximal,
  title={{Maximal functions in Sobolev spaces}},
  author={Aalto, Daniel and Kinnunen, Juha},
  booktitle={Sobolev Spaces In Mathematics I: Sobolev Type Inequalities},
  pages={25--67},
  year={2009},
  publisher={Springer}
}

@book{grafakos2008classical,
  title={Classical fourier analysis},
  author={Grafakos, Loukas and others},
  volume={2},
  publisher={Springer}
}

@book {MajdaBertozzi,
    AUTHOR = {Majda, Andrew J. and Bertozzi, Andrea L.},
     TITLE = {Vorticity and incompressible flow},
    SERIES = {Cambridge Texts in Applied Mathematics},
    VOLUME = {27},
 PUBLISHER = {Cambridge University Press, Cambridge},
      YEAR = {2002},
     PAGES = {xii+545},
      ISBN = {0-521-63057-6; 0-521-63948-4},
   MRCLASS = {76-02 (35Q30 35Q35 76B03 76D03 76D05)},
  MRNUMBER = {1867882},
MRREVIEWER = {Yuxi\ Zheng},
}

@article{starovoitov1994uniqueness,
  title={Uniqueness of a solution to the problem of evolution of a point vortex},
  author={Starovoitov, Victor Nikolayevich},
  journal={Siberian Mathematical Journal},
  volume={35},
  number={3},
  pages={625--630},
  year={1994},
  publisher={Springer}
}

@article{galeati20262d,
  title={{The 2D Euler equations are well-posed for generic initial data in $L^2$}},
  author={Galeati, Lucio},
  journal={arXiv preprint arXiv:2604.14100},
  year={2026}
}

@article{glass2016motion,
  title={On the motion of a small light body immersed in a two dimensional incompressible perfect fluid with vorticity},
  author={Glass, Olivier and Lacave, Christophe and Sueur, Franck},
  journal={Communications in Mathematical Physics},
  volume={341},
  number={3},
  pages={1015--1065},
  year={2016},
  publisher={Springer}
}

@article{mcowen2022perfect,
  title={{Perfect fluid flows on $\mathbb{R}^d$ with growth/decay conditions at infinity}},
  author={McOwen, Robert and Topalov, Peter},
  journal={Mathematische Annalen},
  volume={383},
  number={3},
  pages={1451--1488},
  year={2022},
  publisher={Springer}
}

@article{cobb2026unbounded,
  title={{Unbounded Yudovich Solutions of the Euler Equations}},
  author={Cobb, Dimitri and Koch, Herbert},
  journal={Archive for Rational Mechanics and Analysis},
  volume={250},
  number={3},
  pages={45},
  year={2026},
  publisher={Springer}
}

@article{loeper2006uniqueness,
  title={{Uniqueness of the solution to the Vlasov--Poisson system with bounded density}},
  author={Loeper, Gr{\'e}goire},
  journal={Journal de math{\'e}matiques pures et appliqu{\'e}es},
  volume={86},
  number={1},
  pages={68--79},
  year={2006},
  publisher={Elsevier}
}

@article{elgindi2020symmetries,
  title={Symmetries and critical phenomena in fluids},
  author={Elgindi, Tarek M and Jeong, In-Jee},
  journal={Communications on Pure and Applied Mathematics},
  volume={73},
  number={2},
  pages={257--316},
  year={2020},
  publisher={Wiley Online Library}
}

@article{elgindi2023wellposedness,
  title={{Wellposedness and singularity formation beyond the Yudovich class}},
  author={Elgindi, Tarek M and Murray, Ryan W and Said, Ayman R},
  journal={arXiv preprint arXiv:2312.17610, to appear in JEMS},
  year={2023}
}

@article{jeong2024logarithmic,
  title={{Logarithmic spirals in 2D perfect fluids}},
  author={Jeong, In-Jee and Said, Ayman R},
  journal={Journal de l’{\'E}cole polytechnique—Math{\'e}matiques},
  volume={11},
  pages={655--682},
  year={2024}
}

@article{lacave2019euler,
  title={{The Euler Equations in Planar Domains with Corners}},
  author={Lacave, Christophe and Zlato{\v{s}}, Andrej},
  journal={Archive for Rational Mechanics and Analysis},
  volume={234},
  number={1},
  pages={57--79},
  year={2019},
  publisher={Springer}
}

@article{han2021euler,
  title={{Euler Equations on General Planar Domains}},
  author={Han, Zonglin and Zlato{\v{s}}, Andrej},
  journal={Annals of PDE},
  volume={7},
  number={2},
  pages={20},
  year={2021},
  publisher={Springer}
}

@article{crippa,
title = {{Lagrangian solutions to the Vlasov-Poisson system with a point charge}},
journal = {Kinetic and Related Models},
volume = {11},
number = {6},
pages = {1277-1299},
year = {2018},
issn = {1937-5093},
doi = {10.3934/krm.2018050},
author = {Gianluca Crippa and Silvia Ligabue and Chiara Saffirio}
}

@article{crippa2,
title = {Flows of vector fields with point singularities and the vortex-wave system},
journal = {Discrete and Continuous Dynamical Systems},
volume = {36},
number = {5},
pages = {2405-2417},
year = {2016},
issn = {1078-0947},
doi = {10.3934/dcds.2016.36.2405},
url = {https://www.aimsciences.org/article/id/820b66a3-4e25-4212-9795-3d297a6115bb},
author = {Gianluca  Crippa and Milton C. Lopes  Filho and Evelyne  Miot and Helena J. Nussenzveig  Lopes}
}

@book{mayza,
  title={Sobolev spaces},
  author={Maz'ya, Vladimir},
  year={2013},
  publisher={Springer}
}

\end{document}